\documentclass[10pt,a4paper]{amsart}
\usepackage[utf8]{inputenc}
\usepackage{amssymb}
\usepackage{hyperref}
\hypersetup{
	colorlinks=true,
	linktocpage=true,
	linkcolor=blue,
	citecolor=blue,
	urlcolor=blue,
	pdftitle={Foliated affine and projective structures},
	pdfauthor={Bertrand Deroin and Adolfo Guillot}
}

\newcommand{\del}[1]{\frac{\partial}{\partial #1}}
\newcommand{\indel}[1]{\partial/\partial #1}

\newtheorem{thm}{Theorem}[section]
\newtheorem{lemma}[thm]{Lemma}
\newtheorem{prop}[thm]{Proposition}
\newtheorem{corollary}[thm]{Corollary}

\theoremstyle{remark}
\newtheorem{rmk}[thm]{Remark}
\newtheorem{exam}[thm]{Example}
\newtheorem{defi}[thm]{Definition}

\numberwithin{equation}{section}

\begin{document}

\title{Foliated affine and projective structures}

\author{Bertrand Deroin}
\address{CNRS-Laboratoire AGM-Universit\'e de Cergy-Pontoise}
\email{bertrand.deroin@u-cergy.fr}

\author{Adolfo Guillot}
\address{Instituto de Matem\'aticas, Universidad Nacional Aut\'onoma de M\'exico, Ciudad Universitaria  04510,
Ciudad de M\'exico \\  Mexico}
\email{adolfo.guillot@im.unam.mx}

\subjclass[2010]{Primary 37F75, 53C12, 	57M50}

\keywords{Holomorphic foliation, projective structure, affine structure, Kodaira fibration}


\begin{abstract} We formalize the concepts  of holomorphic affine and projective structures  along the leaves of holomorphic foliations by curves on  complex manifolds. We show that many foliations admit  such structures, we provide local normal forms  for them at singular points of the foliation,  and we prove some index formulae in the case where the ambient manifold is compact. As a consequence of these, we establish that a regular foliation of general type on a compact algebraic manifold of even dimension does not admit a foliated projective structure. Finally, we  classify foliated affine and projective structures along regular foliations on compact complex surfaces.
\end{abstract}

\maketitle


\section{Introduction}

For a one-dimensional holomorphic foliation on a complex manifold, a \emph{foliated projective structure} is a family of complex projective structures along the leaves of the foliation that vary holomorphically in the transverse direction. Particular cases of such structures are foliated translation structures, corresponding to global holomorphic vector fields tangent to the foliation and vanishing only at its singular points, and foliated affine structures, which are a key tool in the study of holomorphic vector fields without multivalued solutions carried out in~\cite{guillot-rebelo}. Some interesting families  of foliated projective structures are the isomonodromic foliations on moduli spaces of branched projective structures~\cite{Veech, McMullen, CDF, GP}; foliated projective structures also appear prominently in Zhao's classification of birational Kleinian groups \cite{zhao-bir}. As we will see, there are plenty more examples, and it seems that a theory deserves to be developed.  The aim of this article is to begin a systematic study both of these structures and of the closely related affine ones. It concerns chiefly the problems  of the existence of such structures on  compact foliated manifolds, of their local description at the singular points of the foliation,  and  of the relations of their local invariants with the global topology of the foliation and the manifold.

In a manifold of dimension~$n$,  a foliation~$\mathcal{F}$,  in a neighborhood of a  singular point~$p$,   may be defined by a vector field~$Z$ with singular set of codimension at least two, unique up to multiplication by a non-vanishing holomorphic function;  the projectivization of the linear part of~$Z$ at~$p$ is a local invariant of~$\mathcal{F}$, but the linear part in itself is not. In the presence of a generic foliated projective structure, a distinguished linear part of a vector field tangent to~$\mathcal{F}$ at~$p$ may be defined up to   sign: the  eigenvalues of this linear part become the \emph{ramification indices} $\nu_1, \ldots, \nu_n$ of the structure at~$p$ (Section~\ref{sec:norforproj}); they are well-defined up to ordering and up to a simultaneous change of sign (the ambiguity of the sign can be lifted for foliated affine structures). When the vector field is non-degenerate and linearizable with semi-simple linear part, these indices encode the ``cone angles'' induced by the foliated projective structure on each one of the~$n$ separatrices. Our first result, Theorem~\ref{thm:loc:proj}, affirms that, generically, a foliated projective structure in the neighborhood of a singular point  is determined by these indices.

Another result, an index theorem,  concerns the global properties of foliated projective structures, assuming that the ambient   manifold~$M$ is compact  and that both the singular points of the foliation~$\mathcal{F}$ and the foliated projective structure at them  satisfy a non-degeneracy condition. In Theorem~\ref{ourindthmproj}  we prove that, given a symmetric homogeneous polynomial \( \varphi(x_1, \ldots, x_{n+1} ) \) of degree \(n+1\), if \(\varphi_{\mathrm{odd}}\) denotes the odd part of \(\varphi\) in the variable \(x_{n+1}\), the quantity 
\[ \sum _{p\in \text{sing} (\mathcal F) } \frac{\varphi_{\mathrm{odd}} (\nu_1, \ldots , \nu_n, 1) }{\nu_1 \ldots \nu_n} \]
can be expressed as an explicit polynomial in the Chern classes of \( T M\) and \( T_{\mathcal F}\).  
Some instances of our result are the following ones (Examples~\ref{thmexasimplest} and~\ref{thmexase}): if~$n$, the dimension of~$M$, is odd, $n=2k+1$, 
\begin{equation}\label{for:projodd} \sum_{p\in \mathrm{Sing}(\mathcal{F})} \frac{\nu_1+\cdots+\nu_n}{\nu_1 \cdots \nu_n}=  c_1^{2k}(T_\mathcal{F})c_1(TM-T_\mathcal{F})  ;\end{equation}
and if it is even, 
\begin{equation}\label{for:projeven} \sum_{p\in \mathrm{Sing}(\mathcal{F})} \frac{1}{\nu_1 \cdots \nu_n}= c_1^{n}(T_\mathcal{F}).\end{equation}
This last formula only makes sense if~$n$ is even, but if the projective structure  reduces to an affine one, it is also valid for~$n$ odd (Theorem~\ref{ourdescartes}).  As usual, in these formulae, the left hand side vanishes if the singular set is empty. 

For all this to be of interest, we need to have a good knowledge of the foliations  which admit foliated affine and projective structures. 

Foliated affine structures are quite common. Some foliations can be shown to admit them almost by construction (like for the ``evident'' foliations on Inoue or Hopf surfaces, see Examples~\ref{exam-inoue} and~\ref{exam-hopf}), or because they admit a description that makes this patent, like elliptic fibrations  (Example~\ref{ex: elliptic fibrations}) or foliations on complex projective spaces (Example~\ref{exam-cpn}). In general, to a foliation~$\mathcal{F}$ on the manifold~$M$ corresponds a class~$\alpha_\mathcal{F}$ in  $H^1(M,K_\mathcal{F})$  that measures the obstruction for~$\mathcal{F}$ to admit a foliated affine structure (Section~\ref{sec:obst-aff}). There are situations where this group is altogether trivial; this allows, for instance, to prove that all foliations whose normal bundle is ample carry a foliated affine structure  (Lemma~\ref{l: criterion for affine foliated structures}; by adjunction, given a  sufficiently positive line bundle on the manifold~\(M\), every foliation on~$M$ having it for its cotangent bundle supports a foliated affine structure).   Further instances of foliated affine structures may be given by constructing  foliated   connections on some line bundles, and propagating them to the tangent bundle of the foliation~(Section~\ref{sec:otherfol}). This allows, for instance,  to prove that any foliation on a Calabi-Yau manifold (Corollary~\ref{cor:tor-k2}) or on a generic hypersurface of~$\mathbf{P}^3$ (Example~\ref{ex:hyp-p3}) admits a foliated affine structure.

Foliated projective structures can be directly shown to exist in some cases, like on Hilbert foliations, suspensions and  turbulent foliations (Examples~\ref{ex:susp} and~\ref{ex:turb}). In a way similar to the affine one, for a given foliation~$\mathcal{F}$ there is a class~$\beta_\mathcal{F}$
in~$H^1(M,K^2_\mathcal{F})$ which vanishes if and only if it admits a foliated projective structure. This obstruction may be calculated in some cases, and, for instance, this allows to show that every foliation in the product of a curve with~$\mathbf{P}^1$ admits a foliated projective structure, while not always a foliated affine one (Proposition~\ref{p: foliations on products}). 

Despite these positive results, there exist   foliations that do not support any foliated projective structure. Zhao recently established that the universal curve of genus \(g\geq 2\), considered as a foliation, does not carry a foliated projective structure~\cite{zhao}. His argument also shows that this is actually true for any Kodaira fibration (a non-isotrivial, holomorphic one). Further obstructions for the existence of foliated projective structures follow from our index formulae. For instance, if a regular foliation on a complex compact surface admits a foliated projective structure, its signature vanishes~(Corollary~\ref{zerosign}). This gives an alternative proof of the non-existence of a foliated projective structure on a Kodaira fibration, but, more importantly, it implies, through Brunella's classification of regular foliations on surfaces~\cite{brunella-nonsing},  that if  a regular foliation on a surface of general type which is not a fibration admits a foliated projective structure, the surface is a quotient of the bidisk, with the foliation being either the vertical or horizontal one. (On its turn, this last result constitutes a key ingredient in Zhao's classification of birational Kleinian groups in dimension two~\cite{zhao-bir}.) Another consequence of our index theorem is that, more generally, in even dimensions, regular foliations of general type do not support foliated projective structures (Proposition~\ref{typgen-noproj}).

These results   allow us to  fully  classify foliated affine and projective structures along regular foliations on surfaces (Corollary~\ref{c: classification}, Section~\ref{sec:class2dreg}).
  
There are some situations that are closely related to the ones  discussed here, but which do not fall within the scope of this article. The structures we consider   are defined on the actual manifold, and not on an infinite cover of it, as in Griffith's work on the uniformization of Zariski open subsets of algebraic varieties~\cite{Griffiths} (which uses foliated projective structures along a covering of a pencil) or in the ``covering tubes'' of a foliation, as in Ilyashenko's notion of \emph{simultaneous uniformization} (see~\cite{ily} and references therein). Holomorphic foliations by curves  which are hyperbolic as Riemann surfaces carry naturally a leafwise hyperbolic (hence projective) structure; the hyperbolic metric varies continuously in the transverse direction~\cite{verjovsky} (even in the presence of singular points~\cite{linsneto-sim, candel-gm}) and, moreover, plurisubharmonically~\cite{brunella-pluri}, but the leafwise hyperbolic geometry will very seldomly give a foliated projective structure in the sense we  consider here.

We assume that the reader is familiar with both the local and global theory of foliations by curves on complex manifolds, like the material covered in the first chapters of~\cite{brunella}.

We thank ShengYuan Zhao and Serge Cantat for stimulating conversations on this topic, and Omar Antolín, S\'ebastien Boucksom, Sorin Dumitrescu and Fernando Sanz for pointing out helpful references. B.D. is grateful to IMPA and Universidad Federal Fluminense for the stimulating working conditions where part of this work was developed. A.G. thanks the hospitality of the \'Ecole Normale Sup\'erieure de Paris during the sabbatical leave where this work began; he gratefully acknowledges support from grant PAPIIT-IN102518 (UNAM, Mexico).

\section{Definitions and the problem of existence} 
We recall  the notions of  affine and projective structures  on curves, and define similar notions for singular holomorphic foliations by curves. We also give various existence criteria showing that many foliations carry such structures, and examples of foliations that do not.

\subsection{Foliated affine structures} \label{ss: foliated affine structure}

An \emph{affine structure} on a curve is an atlas for its complex structure taking values in~$\mathbf{C}$ whose changes of coordinates lie within the affine group~$\{z\mapsto az+b\}$.

The \emph{affine distortion}  of a local biholomorphism between open subsets of \({\bf C}\) is the operator 
\[\mathcal L (f) := \frac{f'' }{f'  }dz, \]
which plays a fundamental role in the study of affine structures. It vanishes precisely when \(f\) is an affine map. A simple computation shows that,  
for the composition of two germs of biholomorphisms between open sets of \(\mathbf C\),  
\begin{equation}\label{eq:coclogder} \mathcal L ( f\circ g ) = \mathcal L(g) + g^* \mathcal L(f).  \end{equation}
Hence, the affine distortion of a biholomorphism between open subsets of curves equipped with affine structures does not depend on the chosen affine charts.    
Given two affine structures on a curve~$C$,  the affine distortion of the identity map measured in the corresponding affine charts, namely the one-form \( \mathcal L ( \psi\circ\phi^{-1}) \) for \(\phi \) a chart of the first affine structure and \(\psi\) a chart of the second, gives a globally well-defined one-form on~$C$ which vanishes if and only if the affine structures agree. Reciprocally, given an affine structure and a one-form \(\alpha \) on~$C$, if~$\alpha$ reads~$a(z)dz$ in some affine chart of the affine structure, the maps  given in this chart  by the solutions~$\psi$ of~$\psi''= a  \psi'$  give a second globally-defined  affine structure on~$C$. An easy consequence of equation \eqref{eq:coclogder} is that this provides the moduli space of affine structures on \(C\) with the structure of an affine space directed by the vector space of holomorphic one-forms on~\(C\).

Given an affine structure on a curve, the family of vector fields which are constant in the coordinates of the affine structure is well-defined. Such a family is the one of flat sections of a holomorphic connection on the tangent bundle of~\(C\). Reciprocally, given a holomorphic connection on the tangent bundle of the curve, one can define the atlas of charts where the flat sections of the connection are constant vector fields. A change  of coordinates of this atlas maps a  constant vector field  to another constant vector field, and hence it belongs to the affine group. We thus retrieve an affine complex structure on the curve. We deduce that there is a canonical correspondence between affine structures on a curve and connections on its tangent bundle. In particular, the only compact curves admitting affine structures are elliptic ones (see~\cite{benzecri}, or Theorem~\ref{ourdescartes}   in Section~\ref{sec:affindex}). On such a curve, there is a canonical affine structure coming from its uniformization by \( {\mathbf C}\), or, equivalently, by the integration of a given non-identically zero holomorphic one-form.

We will adopt both points of view in order to extend the definition of affine structures on curves to the context of unidimensional singular holomorphic foliations on complex manifolds.  

\subsubsection{The foliated setting}

Let us begin by recalling that a \emph{singular holomorphic foliation \(\mathcal F\) of dimension one on a complex manifold \(M\)} is defined by the data of a covering   by open sets \(\{U_i\}_{i\in I}\) of \(M\)  and a family \(\{Z_i\}_{i\in I}\) of holomorphic vector fields \(Z_i\) on \(U_i\), such that the vanishing locus of~\(Z_i\) in \(U_i\) has codimension at least two, and that on the intersection~\(U_i\cap U_j\) of two open sets of the covering, the   vector fields \(Z_i\) and \(Z_j\) are proportional, namely,  $Z_i=g_{ij}Z_j$ for a function $g_{ij} :U_i \cap U_j\rightarrow {\mathbf C} ^*$. Two such data are regarded as giving equivalent foliations if the subsheafs of the sheaf of sections of the tangent bundle \(TM\) of \(M\) generated by the vector fields \(Z_i\) are the same. This subsheaf is called the \emph{tangent sheaf} of the foliation; it is locally free, and corresponds to the sheaf of sections of a holomorphic line bundle, the \emph{tangent bundle} of the foliation, that we will denote by~\(T_{\mathcal F}\).  We then have a morphism \( T_{\mathcal F} \rightarrow T M \), which vanishes only over the singular set of~\(\mathcal{F}\). This map completely characterizes the foliation \(\mathcal F\), and can be used as an alternative definition of a foliation.  The \emph{canonical} bundle of the foliation is the bundle \( K_{\mathcal F} :=  T^* _{\mathcal F}\).

A first definition of a \emph{foliated affine structure} is the following. 

\begin{defi} 
Let~$M$ be a complex manifold,  and  $\mathcal{F}$ a singular holomorphic  foliation by curves on~$M$. A \emph{holomorphic foliated affine   structure} on \(\mathcal F\)  is an open cover~$\{U_i\}$ of~$M\setminus \mathrm{Sing}(\mathcal{F})$ and submersions~$\phi_i:U_i\to \mathbf{C}$   transverse to~$\mathcal{F}$ such  that, in restriction to a leaf~$L$ of~$\mathcal{F}$, $(\phi_i|_L)\circ(\phi_j|_L)^{-1}$ is an affine map of~$\mathbf{C}$.  
\end{defi}

In this definition, no condition is explicitly imposed on the singular set of the foliation.  The affine geometry of the leaves as they approach it will be studied in Section~\ref{sec:local}.

There are foliations without any foliated affine structure, e.g., those having a compact leaf of genus different from one. Notwithstanding, and in contrast with the scarcity of curves having affine structures, there are many foliations that support them.

\begin{exam} A holomorphic vector field with isolated singularities on a manifold of dimension at least two,  e.g.  a holomorphic vector field on a compact K\"ahler manifold~\cite{kob}, induces a foliated affine structure whose changes of coordinates are not only affine but are actually translations (we will call these  \emph{foliated translation structures}).  
\end{exam}

\begin{exam}\label{exam-cpn}
The orbits of a homogeneous polynomial vector field on~$\mathbf{C}^{n+1}$ are preserved by homotheties, and  the vector field defines a foliation on~\(\mathbf{P}^n\). Not only does the vector field define a foliation; it also endows it with  a foliated affine structure: the homogeneous vector field induces a  translation structure along its phase curves, and multiplication by scalars acts affinely in the translation charts. One can show that, on~\(\mathbf{P}^n\), any foliation  and any foliated affine structure on it are obtained in this way. \end{exam}

\begin{exam}[Inoue surfaces]\label{exam-inoue} The Inoue surfaces~$S_M$, $S_N^{(+)}$ and~$S_N^{(-)}$  are compact complex non-K\"ahler surfaces which are quotients of~$\mathbf{H}\times \mathbf{C}$ by groups of affine transformations of~$\mathbf{C}^2$~\cite{inoue}. For the surfaces~$S_M$, the associated action preserves the two foliations of~$\mathbf{H}\times \mathbf{C}$ and is affine on the leaves of both of them (with respect to the tautological  affine structure on~$\mathbf{C}$, to the one inherited from the inclusion~$\mathbf{H}\subset\mathbf{C}$ for~$\mathbf{H}$). The two foliations induced in~$S_M$ admit thus foliated affine structures. For the surfaces~$S_N^{(+)}$ and~$S_N^{(-)}$, the action on~$\mathbf{H}\times \mathbf{C}$ preserves the foliation given by the fibers of the first factor, and acts affinely upon its leaves. The induced foliations are also endowed with foliated affine structures. 
\end{exam}

\begin{exam}[Hopf surfaces] \label{exam-hopf} Hopf surfaces are compact complex surfaces whose universal covering is biholomorphic to~$\mathbf{C}^2\setminus\{0\}$~\cite[Ch.~V, Section~18]{bphv}. Primary ones are quotients of~$\mathbf{C}^2\setminus\{0\}$ by cyclic groups  generated by contractions  which admit the normal form 
\begin{equation}\label{eq:lochopf} (x,y)\mapsto (\alpha x+\lambda y^n,\beta y),\end{equation}	
with~$(\alpha-\beta^n)\lambda=0$, $n\geq 1$. If~$\alpha=\beta$, the linear homogeneous vector fields are preserved by the contraction and induce in the quotient  nowhere-vanishing vector fields; their foliations   have thus a foliated translation  structure. In the general case, the same happens for the linear diagonal vector fields~$Ax\indel{x}+By\indel{y}$ if~$\lambda=0$, or the ``Poincaré-Dulac'' ones $(nx+\mu y^n)\indel{x}+y\indel{y}$  if~$\lambda\neq 0$.
	
The coordinate vector field~$\indel{x}$ is not fixed by the contraction, but is preserved up to a constant factor, and the  foliation it induces has thus an invariant foliated affine structure;  both the foliation and the structure descend to the quotient. (In particular cases, the same happens for other constant vector fields, but this case reduces to the previous one by means of a linear change of coordinates.)

By~\cite{brunella-nonsing}, there are no further  foliations on primary Hopf surfaces.  Consequently,  \emph{every foliation on a primary Hopf surface has a foliated affine structure}.  
	
Secondary Hopf surfaces are unramified quotients of primary ones, quotients of~$\mathbf{C}^2\setminus\{0\}$ by the action of the semidirect product of the infinite cyclic group~$G$ generating the associated primary  surface and a finite group~$H$ that normalizes it. In what follows we rely on the classification due to Kato~\cite{kato-hopf, kato-hopf-err}. There are coordinates where~$G$ is generated  by~(\ref{eq:lochopf}) and where~$H$ is a subgroup of~$\mathrm{GL}(2,\mathbf{C})$. A foliation on the secondary Hopf surface  is induced by a vector field  generating the foliation of the primary one, and which, from Kato's classification, is easily seen to be preserved by~$H$ up to a constant factor.  Thus, \emph{every foliation on a secondary Hopf surface has a foliated affine structure as well}.  \end{exam}

\subsubsection{Foliated connections}

Let us now turn to a more intrinsic equivalent definition of a foliated affine structure  in terms of foliated connections,   which will  enable the construction of more examples. 

Given a foliation~$\mathcal{F}$ and a sheaf  \(\mathcal S\) of \(\mathcal O_S\)-modules, a \emph{foliated connection  on \(\mathcal S\) relative to~$\mathcal{F}$} is a differential operator \( \nabla:\mathcal S \rightarrow \mathcal O (K_{\mathcal F}) \otimes \mathcal S\)  which satisfies the Leibniz rule 
\begin{equation} \label{eq: Leibniz} \nabla ( f s ) = d_{\mathcal F} f \otimes s + f \nabla s , \end{equation}
for every~\( f\in \mathcal O\) and every~\(s\in \mathcal S\). (In general, we will consider~$\mathcal{F}$ as fixed, and omit it from the discussion.) A foliated connection on a holomorphic vector bundle is a foliated connection on its sheaf of sections.

In particular,  a foliated connection on~$T_\mathcal{F}$ is  a map
\( \nabla : T_{\mathcal F} \rightarrow T_{\mathcal F} \otimes K_{\mathcal F} = \mathcal O _M\)
which  to a vector field~$Z$ assigns a holomorphic  function~$\nabla(Z)$, its \emph{Christoffel symbol},  satisfying the Leibniz rule
\begin{equation}\label{conn-fun}\nabla(fZ)=Zf+f\nabla(Z).\end{equation}
Let us see that such a connection is equivalent to a foliated affine structure. 

Given a  foliated connection~$\nabla$ on~$T_\mathcal{F}$ and~$p\notin\mathrm{Sing}$, if~$Z$ is a vector field tangent to~$\mathcal{F}$ that does not vanish at~$p$ and such that~$\nabla(Z)\equiv 0$ (if~$Z$ is \emph{parallel}), if~$\phi$ is a function  such that~$d\phi(Z)\equiv 1$, $\phi$ is part of an atlas of a foliated affine structure  that depends only on~$\nabla$ (it is not difficult to see that such a~$Z$ and such a~$\phi$ always exist).  

For the other direction,  let~$\mathcal{F}$ be a foliation  endowed with a   foliated affine structure~$\sigma_0$. Let~$Z$ be a vector  field defined on the open set~$U\subset M$, tangent to~$\mathcal{F}$ (with a singular set of codimension two), and denote by~$\sigma_Z$ the foliated affine structure induced by~$Z$ in~$U\setminus\mathrm{Sing}(\mathcal{F})$. The difference~$\sigma_Z-\sigma_0$ is a section~$\alpha$ of~$K_\mathcal{F}$ over~$U$, which vanishes if and only if~$\sigma_0$ is induced by~$Z$. The holomorphic function~$\alpha(Z)$, defined, in principle, in~$U\setminus\mathrm{Sing}(\mathcal{F})$, extends, by Hartog's theorem, to all of~$U$. Define a  foliated connection~$\nabla$ on~$T_\mathcal{F}$ by $\nabla(Z)=\alpha(Z)$. Let us verify that Leibniz's rule~(\ref{conn-fun}) takes place. We will do so locally in a curve, in a coordinate~$z$ where~$Z=\indel{z}$. A chart for  the affine structure induced by~$f(z)\indel{z}$ is~$\int^z d\xi/f(\xi)$, and thus
\[  \sigma_{Z}-\sigma_{fZ}= -\frac{f'}{f}dz.\]
Hence, the contraction of~$\sigma_{fZ}-\sigma_0$ with~$fZ$ yields $Zf+f\alpha(Z)$, in agreement with formula~(\ref{conn-fun}).

Observe that the definition of  foliated affine structures via   foliated connections  has the advantage of not needing to distinguish between regular and singular points of the foliation.

\begin{lemma}[Extension Lemma {\cite[Prop.~8]{guillot-rebelo}}]  \label{l: extension lemma} Let~$M$ be a manifold, $\mathcal{F}$   a foliation on~$M$, $p\in M\setminus\mathrm{Sing}(\mathcal{F})$. Let~$X$ be a meromorphic vector field defined in a neighborhood of~$p$ whose divisor of zeros and poles~$D$ is invariant by~$\mathcal{F}$ and which is tangent to~$\mathcal{F}$ away from it. Then, in a neighborhood of~$p$, the foliated affine structure induced by~$X$ away from~$D$ extends to~$D$ in a unique way. \end{lemma} 
\begin{proof}  	Let~$Z$ be a non-vanishing holomorphic vector field defining~$\mathcal{F}$ in a neighborhood of~$p$. Let~$X=f_1^{n_1}\cdots f_k^{n_k}Z$, for $n_i\in\mathbf{Z}$ and reduced holomorphic functions~$f_i$ such that~$Zf_i$ divides~$f_i$, say~$Zf_i=h_if_i$ for some holomorphic function~$h_i$. Outside the divisor of zeros and poles of~$X$, $X$ induces a foliated connection~$\nabla$ such that~$\nabla(X)\equiv 0$, for which
\[\nabla(Z)=\nabla\left(\left(\prod f_i^{-n_i}\right)X\right)=\left(\prod f_i^{n_i}\right)Z\left(\prod f_i^{-n_i}\right)=-\sum_{i=1}^k \frac{n_i}{f_i} Zf_i =-\sum_{i=1}^k  n_i h_i,\]
and~$\nabla$ extends holomorphically to a full neighborhood of~$p$. \end{proof}	
 
Let us see how one can concretely apply this lemma to produce foliated affine structures.

\begin{exam}[Elliptic fibrations] \label{ex: elliptic fibrations} Every elliptic fibration with non-singular fibers admits a foliated affine structure.  Indeed, a smooth elliptic curve carries a canonical affine structure, which varies holomorphically with the elliptic curve.  This structure can be extended to the singular fibers. In the case of the universal elliptic curve, this is just a corollary of the fact that the Hodge bundle of abelian differentials on the fibers exists~\cite{Zvonkine}. (The fact  that the Chern class of this bundle does not vanish shows that one cannot reduce this affine structure to a translation structure.) In the general case one can argue as follows: first build a non-identically zero meromorphic section of \( T_{\mathcal F}\)  whose divisor  of zeros and poles is supported on a union of fibers. To do so, one can apply Corollary~12.3 in~\cite[Ch.~V]{bphv} to get a meromorphic volume form~$\omega$ on the total space  whose divisor of zeros and poles is supported on a finite union of fibers, the desired section of \(T_{\mathcal F}\) being the gradient vector field  of a meromorphic function defined on the base with respect to~$\omega$. On the fibers on which this vector field is regular, it induces the canonical affine structure.  Lemma \ref{l: extension lemma} then shows that this structure extends to a globally defined foliated affine structure.  
\end{exam}

\subsubsection{A cohomological obstruction}\label{sec:obst-aff} 
	
For general fibered spaces, there is a classical cohomological obstruction for the existence of a connection. In our setting,  there is a natural class  $\alpha_{\mathcal F}$ in $ H^1 (M ,K_{\mathcal F})$, whose vanishing is equivalent to the existence of a foliated connection on \( T_{\mathcal F}\), or, equivalently, of a foliated affine structure on~\(\mathcal F\). 	Let us recall this construction in our case. Observe that, locally, foliated affine structures exist, e.g. the translation structures associated to  vector fields generating~\(\mathcal F\). Let \( \{U_i\}_{i\in I}\) be a covering by open sets of \(M\) so  that  a foliated affine connection on~$T_\mathcal{F}$, $\nabla_i$, is defined on each \(U_i\). In the intersection \(U_i\cap U_j\), the difference~$\nabla_i-\nabla_j$ is a section~\(\alpha_{ij}\)   of \(K_\mathcal F\) on \(U_i\cap U_j\). Moreover,   \( (\alpha_{ij})_{ij}\) is a cocycle. It is easy to see that the cohomology class   \(\alpha_\mathcal{F}\)  in \( H^1(M,K_{\mathcal F})\) induced by \((\alpha_{ij})_{ij}\) does not depend on the choices made. To construct a globally defined   connection,   we need  to modify each affine connection \(\nabla_i\) on~\(U_i\) by the addition of a section~\(\alpha_i\)  of~\(K_\mathcal F\), \(\nabla_i'=\nabla_i+\alpha_i\), so that the  \(\nabla_i'\)'s   coincide on the intersection of their domains. This means that on~\(U_i\cap U_j\), \( \alpha_i - \alpha_j= \alpha_{ij}\), which amounts to saying that the class \( \alpha_\mathcal{F} \) in \( H^1 (M, K_{\mathcal F})\)  is trivial. Hence, \emph{a foliated affine structure exists if and only if the class \(\alpha_{\mathcal F}\) vanishes}.  This discussion also shows that, if non empty, \emph{the moduli space of foliated affine structures on a given singular holomorphic foliation~$\mathcal{F}$ is an affine space directed by the space \(H^0(K_{\mathcal F})\)}.  In particular, a foliated affine structure is unique whenever the canonical bundle does not have a non-identically zero section.  We  exhibit some foliations supporting unique foliated affine structures in Section~\ref{sec:class2dreg}.

We next derive a criterion for the existence of foliated affine structures which is a consequence of Kodaira's vanishing theorem and Serre's duality:

\begin{lemma}\label{l: criterion for affine foliated structures} 
Assume that \(\mathcal F\) is a singular holomorphic foliation by curves on a compact manifold~$M$ of dimension \(n>1\). If either \(T_{\mathcal F}\) or  \(\wedge ^{n-1} \mathcal{N}_{\mathcal F} \) is ample, then \(\mathcal F\) carries a holomorphic foliated affine structure.
\end{lemma}

\begin{proof} Notice that since~$M$ has an ample line bundle, it is projective. 
Recall the Kodaira vanishing theorem: given an ample divisor \(D\) on \(M\), \(H^q (K_M+D) =0\) for any \(q>0\).  By the adjunction formula,  $K_M + \wedge^{n-1} \mathcal{N}_{\mathcal F}= K_{\mathcal F}$, so if~$\wedge^{n-1} \mathcal{N}_{\mathcal F}$ is ample,  \( H^1 (K_{\mathcal F})=0\), which implies that \(\mathcal F\) admits a holomorphic foliated affine structure. Now, using Serre's duality, Kodaira's vanishing theorem also furnishes the following: for every ample divisor \(D\), \( H^1 (D^*) =0\). Hence, if \( T_{\mathcal F}\) is ample, the foliation~\(\mathcal F\)  carries a holomorphic foliated affine structure as well.  \end{proof}

Notice that under the assumption that \(T_{\mathcal F}\) is ample, a theorem of Miyaoka (see~\cite{Miyaoka}, \cite[Ch.~7]{brunella})  implies that the foliation \(\mathcal F\) is a pencil by rational curves.

\subsubsection{More on foliated connections} \label{sec:otherfol}

For a given foliation~$\mathcal{F}$, for the problem of establishing the existence of  foliated connections on~$T_\mathcal{F}$, investigating the existence of foliated connections on other line bundles might prove rewarding, since the set of isomorphism classes of line bundles admitting foliated connections forms a group and is closed under the operations of taking powers and extracting roots:  foliated connections on other line bundles might propagate up to~$T_\mathcal{F}$.    An interesting problem is thus that of determining, for a singular holomorphic foliation on a complex manifold, which are the holomorphic line bundles having foliated connections. 
	
A fundamental example of a foliated connection is the Bott connection on the conormal sheaf \(\mathcal N^* _{\mathcal F} \)  of a foliation \(\mathcal F\)~\cite{bott}. If we denote by~\(n\) the dimension of~\(S\),   \(\mathcal N^* _ {\mathcal F}\) is the sheaf of holomorphic forms of degree \(n-1\) on \(M\) whose contraction with any vector field  tangent to \(\mathcal F\) vanishes. The conormal bundle is not always locally free, but its maximal exterior power is, and the Bott connection is then the usual derivative operator 
\[ d : \wedge ^{n-1} \mathcal N ^* _{\mathcal F}  \rightarrow \mathcal O (\wedge ^n TM)  \simeq \mathcal O (K_{\mathcal F} ) \otimes \wedge ^{n-1} \mathcal N^* _{\mathcal F},  \]
the last isomorphism being given by the adjunction formula. 

On a closed K\"ahler manifold, every holomorphic line bundle with trivial first Chern class carries a flat unitary connection, which, by restriction, induces a foliated connection. Hence, in this setting, the problem consists in determining which are the Chern classes of line bundles which carry foliated connections. This set is a subgroup of the N\'eron-Severi group which contains all the torsion points.  As we have seen, it contains the first Chern class of the normal sheaf of the foliation, but, in general, it seems difficult to say more. There are, however, situations where this point of view permits to prove the existence of foliated affine structures. Let us give some examples:

\begin{lemma}  
On a compact K\"ahler manifold with vanishing first Chern class, any singular holomorphic foliation carries a foliated affine structure. 
\end{lemma}

\begin{proof}
Since the manifold has vanishing first Chern class, its canonical bundle has a unitary flat connection.  By the adjunction formula, the tensor product of this connection with the Bott connection produces a flat connection on the cotangent bundle of the foliation, and hence, by duality,  a foliated affine structure.\end{proof}

\begin{corollary} \label{cor:tor-k2} Any foliation on a Calabi-Yau manifold has a foliated affine structure. \end{corollary}

\begin{lemma}   
If the \(\mathbf{Q}\)-rank of the N\'eron-Severi group of a compact K\"ahler manifold \(M\)  is one, then 
\begin{itemize}
\item if the first Chern class of the normal sheaf to the foliation is not a torsion element in the N\'eron-Severi group, there is a foliated affine structure; 
\item otherwise, $\mathcal{F}$ has a transverse invariant pluriharmonic form.
\end{itemize}	
\end{lemma}

\begin{proof} Notice that the maximal wedge power of the normal sheaf is locally free, namely a line bundle, equipped with the Bott foliated connection; so if its first Chern class  is not a torsion element in the N\'eron-Severi group, then any line bundle over \(S\) has a foliated connection.   In particular, the tangent bundle carries a foliated affine structure.  	If not, then the highest non trivial wedge power \( \wedge ^{n-1} \mathcal{N}^*_{\mathcal F}\)  of the normal sheaf carries a unitary flat connection over \(S\). Given a flat section \( \omega\),   naturally considered as a holomorphic form of degree \( n-1\), the product \( \omega \wedge \overline{\omega}\) is a well-defined pluriharmonic form on \( S\) which vanishes on the foliation \( \mathcal F\). Such a form is closed because \(S\) is K\"ahler, and hence  defines a family of transverse pluriharmonic forms. 
\end{proof}

\begin{exam}[Hypersurfaces of \({\mathbf P}^3\)] \label{ex:hyp-p3} Well-known examples of surfaces having Picard number one are generic hypersurfaces of \({\mathbf P}^3\) of degree at least four, by a theorem of Noether, see \cite{Deligne}. These  are simply connected by the hyperplane section theorem of Lefschetz, and in particular it is impossible in this case for the normal bundle to a foliation to have a torsion first Chern class. Indeed,  if it were the case, the normal bundle would be holomorphically trivial, and so would be its dual, and consequently we would have a holomorphic form on the surface vanishing on the foliation. However, such a form does not exist since the surface has a vanishing first Betti number. In other words, we have proved that \emph{on a generic surface in~\(\mathbf{P}^3\), every singular holomorphic foliation carries a foliated affine structure.} Notice that this property holds on the explicit examples produced in \cite{Shioda}, namely the surfaces defined in homogeneous coordinates by \( w^m+ xy^{m-1}+y z^{m-1}+zx^{m-1}=0\) for \(m\geq 5\) a prime number.  
\end{exam}

\subsection{Foliated projective structures} \label{ss: foliated projective structure}

A projective structure on a curve is an atlas for its complex structure taking values in~$\mathbf{P}^1$ whose changes of coordinates lie within the group of projective transformations~$\{z\mapsto (az+b)/(cz+d)\}$.
In this case,  the \emph{Schwarzian derivative}
\begin{equation}\label{eq:def-shw}\{f(x),x\}=\frac{f'''}{f'}-\frac{3}{2}\left(\frac{f''}{f'}\right)^2=\left(\frac{f''}{f'}\right)'-\frac{1}{2}\left(\frac{f''}{f'}\right)^2,\end{equation}
plays a role analogous to the one played by the affine distortion in the context of affine structures. 

Given two projective structures on a curve~$C$ with charts~$\{(U_i,\phi_i)\}$ and $\{(V_j,\psi_j)\}$, the quadratic form on~$U_i\cap V_j$ given by
\begin{equation}\label{projoperator}
\{f(z),z\}dz^2,
\end{equation}	
for~$f=\psi_j\circ \phi_i^{-1}$, gives a globally well-defined  quadratic form on~$C$, which vanishes if and only if the projective structures coincide. This is due to the fact that the operator \eqref{projoperator} satisfies
\[ \{f\circ g,z\}dz^2=\{g,z\}dz^2+g^*(\{f,w\}dw^2).\]
Reciprocally, given a projective structure with charts~$\{(U_i,\phi_i)\}$ and a  quadratic form~$\beta$ on~$C$, if~$\beta$ reads~$\beta_i(z)dz$ in~$U_i$, the charts locally given by the solutions of the Schwarzian differential equation~$\{f,z\}=\beta_i$ give a globally well-defined projective structure on~$C$. This is,
on a curve, \emph{the projective structures form an affine space directed by the vector space of holomorphic quadratic differentials}.

Projective structures are much more flexible than affine ones:  they exist on \emph{any} curve, and their moduli is an affine space of dimension \(3g-3\). Projective structures  associated to particular geometries (spherical for genus zero, Euclidean in the case of genus one, and hyperbolic for genus at least two) are given by the Uniformization Theorem~\cite{henri-paul}. Nevertheless, the existence of unrestricted projective structures can be very easily established independently from it, as Poincar\'e was well aware of; see~\cite[\S 9]{gunning} for a modern presentation.

\begin{defi}\label{def:proj} Let~$M$ be a complex manifold, $\mathcal{F}$ a singular holomorphic  foliation by curves on~$M$. A \emph{holomorphic foliated projective   structure} on~$M$ over~$\mathcal{F}$ is an open cover~$\{U_i\}$ of~$M\setminus \mathrm{Sing}(\mathcal{F})$ and submersions~$\phi_i:U_i\to \mathbf{P}^1$   transverse to~$\mathcal{F}$ such  that, in restriction to a leaf~$L$ of~$\mathcal{F}$, $(\phi_i|_L)\circ(\phi_j|_L)^{-1}$ belongs to~$\mathrm{PSL}(2,\mathbf{C})$.
\end{defi}

Foliated projective structures may also be defined in terms of foliated projective connections: a \emph{foliated projective connection} is a map~$\Xi:T_\mathcal{F}\to \mathcal{O}(M)$  that to a vector field~$Z$ associates a holomorphic function~$\Xi(Z)$, its~\emph{Christoffel symbol}, satisfying the modified Leibniz rule 
\begin{equation}\label{christproj}
\Xi(fZ)=f^2\Xi(Z)+fZ^2(f)-\frac{1}{2}(Zf)^2.
\end{equation}
For instance, if~$\nabla:T_\mathcal{F}\to \mathcal{O}(M)$ is a foliated connection on~$T_\mathcal{F}$, the associated projective connection~$\Xi$ is
\begin{equation}\label{for:affaproj} \Xi(Z)=-\frac{1}{2}(\nabla(Z))^2+Z(\nabla(Z)).\end{equation}

Let us see that a foliated  projective structure is equivalent to  a foliated projective connection.  Let~$\mathcal{F}$ be a foliation endowed with a   foliated projective structure~$\rho_0$. Let~$Z$ be a  vector field tangent to~$\mathcal{F}$ with singular set of codimension at least two, and consider the projective structure~$\rho_Z$ that it defines  away from $\mathrm{Sing}(\mathcal{F})$. The difference~$\rho_Z-\rho_0$ is a section~$\alpha$ of $K_\mathcal{F}^2$. Define~$\Xi(Z)$ as~$\alpha(Z^{\otimes 2})$.  It is a holomorphic function in the regular part of~$\mathcal{F}$ and, by Hartog's theorem,  extends to all of~$M$. Let us prove that it satisfies condition~(\ref{christproj}). As before, it is sufficient to do so locally in a curve. Consider a curve  endowed with a projective structure~$\rho_0$, $Z$ a holomorphic vector field and $z$ a local coordinate in which  $Z=\indel{z}$. Let~$\alpha(z)dz^2$ be the quadratic form~$\rho_Z-\rho_0$. The projective structure defined by~$fZ$ has~$\int^zd\xi/f(\xi)$ as a chart and thus 
\[ \rho_Z-\rho_{fZ} =   \left(\frac{1}{2}\left(\frac{f'}{f}\right)^2-\frac{f''}{f}\right)dz^2.\]
Hence, the contraction of~$\rho_{fZ}-\rho_0$ with~$(fZ)^{\otimes 2}$  gives~$f^2\alpha(Z^{\otimes 2})+Z^2f-\frac{1}{2}(Zf)^2$, establishing~(\ref{christproj}).   Reciprocally, if~$\Xi$ is a foliated projective connection, $p\notin\mathrm{Sing}(\mathcal{F})$ and~$Z$ is a holomorphic vector field tangent to~$\mathcal{F}$ that does not vanish at~$p$, and such that~$\Xi(Z)\equiv 0$, if~$\phi$ is a function defined in a neighborhood of~$p$ such that~$d\phi(Z)\equiv 1$, $\phi$ defines a foliated projective structure in the sense of Definition~\ref{def:proj} that depends only on~$\Xi$.

When restricted to curves, our definition of projective connection is equivalent to some  of the existing ones (\cite[Def.~1.3.1]{Tyurin}, \cite[Section~4]{gunning-special}).

\begin{exam}[Suspensions] \label{ex:susp} A non-singular foliation by curves~$\mathcal{F}$ on a compact surface~$S$ is a \emph{suspension} if there exists a fibration onto a curve~$\pi:S\to C$ which is everywhere transverse to~$\mathcal{F}$.  \emph{On a suspension, every foliated  projective structure is the pull-back of a projective structure on the base}. In fact, if~$\Xi_\mathcal{F}$ is a foliated projective connection and~$Z$ is a non-vanishing vector field defined in an  open subset~$U$ of~$C$, if~$\pi_\mathcal{F}^*Z$ denotes the pull-back of~$Z$ tangent to~$\mathcal{F}$ then since~$\Xi_\mathcal{F}(\pi_\mathcal{F}^*Z)$ is a  holomorphic function, it is constant along the fibers of~$\pi$. In this way, the projective structure on~$C$ given by  $\Xi_C(Z):=\Xi_\mathcal{F}(\pi_\mathcal{F}^*Z)$ is well-defined; the foliated one is its pull-back. 
\end{exam}

\begin{exam}[Turbulent foliations admit foliated projective structures] \label{ex:turb}  Let~$S$ be a compact surface, $\pi:S\to C$ an elliptic fibration without singular fibers,  $\mathcal{F}$  a turbulent foliation on~$S$ adapted to~$\pi$, this is, every fiber of~$\pi$ is either tangent or everywhere transverse to~$\mathcal{F}$ (see~\cite[Ch.~4]{brunella}). Let~$C_0\subset C$ be the  subset above which~$\pi$ and~$\mathcal{F}$ are transverse. By the arguments in Example~\ref{ex:susp}, the  projective structures on~$C_0$ and the foliated  ones on~$\pi^{-1}(C_0)$  are in correspondence. Let~$p\in C\setminus C_0$. Let us show that if the projective structure on~$C_0$ has a moderately degenerate behavior at~$p$, the corresponding foliated projective structure  extends to a neighborhood of the fiber above~$p$.  The fibration around~$\pi^{-1}(p)$ is given by~$\mathbf{D}\times E\mapsto \mathbf{D}$ for some elliptic curve~$E$. For some local  coordinates $z$ and~$w$ in~$\mathbf{D}$ and~$E$ ($z$ centered at~$p$),  $\mathcal{F}$ is given by the vector field~$Z=\indel{z}+z^{-n}B \indel{w}$, with~$B$  a holomorphic   non-vanishing function and~$n>0$~\cite[Ch.~4, Section~3]{brunella}. Let~$\Xi_0$ be a projective connection on~$\mathbf{D}\setminus \{0\}$, and let~$\Xi$ be the corresponding foliated projective connection on~$\pi^{-1}(\mathbf{D}\setminus \{0\})$. In the spirit of Lemma~\ref{l: extension lemma}, by formula~(\ref{christproj}), for the holomorphic and non-vanishing vector field~$z^nZ$, since~$\pi_*Z=\indel{z}$,
\begin{equation}\label{ext-conn} \Xi(z^n Z)=z^{2n-2}\left(z^2\Xi_0\left(\del{z}\right)+\frac{1}{2}n(n-2)\right).\end{equation}
If this expression is holomorphic (for instance, if~$\Xi_0(\indel{z})$ has at most a   pole of order two at~$0$) the foliated projective structure extends to the non-transverse fiber. (If we had chosen an affine connection~$\nabla_0$ instead of the projective one~$\Xi_0$, (\ref{ext-conn}) would read~$\nabla(z^nZ)=z^{n-1}(z\nabla_0(\indel{z})+n)$.)
\end{exam}

Not all foliations support foliated projective structures. As we mentioned in the introduction, by the work of Zhao \cite{zhao}, no Kodaira fibration admits one (see also Corollary~\ref{zerosign}  for an alternative proof). Despite the generality of this result, we thought  it worthwhile to include a concrete, hands-on, self-contained instance of it:

\begin{exam}[An explicit Kodaira fibration without a foliated projective structure]\label{thm:kodfib} Recall that a \emph{Kodaira fibration} is a  smooth  holomorphic fibration $S\rightarrow B$ from a complex surface over a curve which is not a holomorphic fiber bundle (Kodaira gave the first examples of such fibrations~\cite{kodaira}, see also \cite[p.~220]{bphv}).  Through a construction  close  to Atiyah's one~\cite{atiyah}, we  here construct an explicit  Kodaira fibration with fibers of genus six which does not support a foliated projective structure, i.e. such that there is no family of projective structures on the fibers varying holomorphically.   
	
Start with a curve $C$ of genus two, and let $\pi : C' \rightarrow C$ be a connected non-ramified double covering ($C'$ has genus three). For every $x\in C$, we construct~$64$ curves of genus~$6$: the ramified double coverings $C'' \rightarrow C'$ ramified over $\pi^{-1}(x)$. Such a covering is determined by a morphism from $H_1(C'\setminus \pi^{-1}(x), \mathbf Z)$ to $\mathbf Z / 2\mathbf Z$ that maps the peripheral cycles around each one of the two punctures to $1$; the number of such coverings is~$64$.  Construct the Kodaira fibration $ F : S\rightarrow B$ by putting all these surfaces over the point $x\in C$, and taking a connected component---we don't know in  general if the resulting surface is connected, a case that would lead to a genus~$65$  base~$B$.  Hence, a point~$y$ in~$B$ is the data of a point $x\in C$ and of a double covering~$\delta_y : F^{-1}(y) \rightarrow C'$ ramified over $\pi^{-1} (x)$. Let~$d$ be the degree of the covering~$B\to C$.

We claim that this fibration does not carry a foliated projective structure. Assume by contradiction that there exists a family of projective structures $\{\sigma_y\} _{y\in B} $ on the fibers of $F$ that vary holomorphically with $y$. Introduce a family of branched projective structures~$\{ \beta_y \}_{y\in B}$ on the fibers $F^{-1}(y)$, $\beta_y$ being the pull-back of a (non-branched) projective structure $\nu$  on $C'$ by $\delta_y$. The Schwarzian derivative of $\beta_y$ in the charts given by $\sigma_y$ gives a family of meromorphic quadratic differentials on $F^{-1}(y)$ that vary holomorphically with the $y$ parameter, and which have poles of order $2$ located at the points $\delta_y^{-1} (\pi^{-1} (x) ) $, with residue~$-3/2$ (as quadratic differentials). Indeed, if $u, v$ are charts of $\sigma$ and $\beta$ at such a point, we have $v = c_2 u^2 + \ldots$ where $c_2\neq 0$, hence~$\{v,u\}=-\frac{3}{2}u^{-2}+\ldots$. 
We denote by $Q_y$ the quadratic differential on $F^{-1} (y)$. 
	
For each $y$ in~$B$, we denote by $i_y$ the involution on $F^{-1}(y)$ that exchanges the points in the fiber of $\delta_y$, and we define $R_y =  Q_y + i_y^* Q_y$. This is an $i_y$-invariant meromorphic quadratic differential on $F^{-1} (y)$ having poles at $\delta_y^{-1} ( \pi^{-1} (x))$ of order~$2$ and residues $-3$. Hence, there is a meromorphic quadratic differential~$S_y$ on $C'$  such that $R_y = \delta_y^* S_y$. This differential has poles on the set $\pi^{-1} (x)$, and is holomorphic elsewhere. We claim that the poles on $\pi^{-1} (x)$ are of order $2$, and that the residues are $-3/4$. To see this, take  coordinates $v,w$ in $F^{-1}(y)$ and in $C'$ respectively such that $\delta_y $ is the map $v\mapsto w= v^2$. The quadratic differential $R_y$ is expressed in the $v$-coordinates by $R_y =  ( -3v^{-2} + c_0 + c_2 v^2 + \ldots ) dv^2$, since it is invariant by the involution $v\mapsto -v$. Hence,   $S_y =  \frac{1}{4}( -3w^{-2} + c_0w^{-1} + c_2 w + \ldots  ) dw^2$, proving the claim. 
	
We now define, for $x\in C$, the meromorphic quadratic differential $T_x$ on $C'$ by~$T_x = \sum S_y$ for all the coverings $y\in B$ corresponding to the point $x$. We see that $T_x$ has poles only at $\pi^{-1} (x)$, that these are of order~$2$ and that the residues are~$-3d/4$.  
Let $j$ the involution on $C'$ which exchanges the fibers of $\pi$, and let $U_x$ be the meromorphic quadratic differential on $C$ which satisfies $\pi ^* U_x = T_x + j^* T_x$.  The family $\{U_x\}_{x\in C}$ is a holomorphic family of meromorphic quadratic differentials on $C$ having a unique pole on $C$ at $x$ of order $2$ and of residue $-3d/2$. 
	
We claim that such a family of meromorphic quadratic differentials cannot exist.  Indeed, choose a point $x_0$ in $C$ which is not fixed by any non-trivial involution  and such that there exists a holomorphic quadratic differential~$\mu$ on~$C$ that does not vanish at~$x_0$.  Consider the holomorphic function~$f:C\setminus\{x_0\}\to \mathbf{C}$ given at~$x$ by the evaluation of~$U_x/\mu$ at~$x_0$.  It extends meromorphically to $x_0$, having there a pole of order $2$, because for a local coordinate~$z$ centered at~$x_0$, $U_x/\mu = -\frac{3}{2}(z-x)^{-2} +\ldots$. Hence, $f$ extends to  a ramified double covering from $C$ to $\mathbf P^1$, and the involution exchanging its fibers fixes $x_0$. But this is a contradiction.\end{exam}

The existence of a foliated projective structure is equivalent to the vanishing of a class \(\beta_{\mathcal F}\) in \(H^1 (M, K_{\mathcal F}^2)\). Its    definition mimics the definition of the class \(\alpha_{\mathcal F}\) introduced in the context of foliated affine structures.  Namely,  take a covering of \(M\) by open sets \(U_i\) on which we have   foliated connections~$\Xi_i$, and consider the cocycle \( \beta =(\beta_{ij}) _{ij} \), where \(\beta_{ij}=\Xi_i-\Xi_j\) is a section of \(K_{\mathcal F}^2\) over~$U_i\cap U_j$. Its cohomology class~$\beta_\mathcal{F}\in H^1 (M,K_{\mathcal F}^2)$ is well-defined. To construct a globally defined foliated projective connection, one needs to modify each \(\Xi_i\) in \(U_i\) by adding some section \(\beta_i\) of \(K_{\mathcal F}^2\), in such a way that the resulting connections on the \(U_i\)'s coincide in the intersection of their domains. This is equivalent to solving the equation \( \beta_i - \beta_j = \beta_{ij}\), so \emph{there exists a foliated projective structure if and only if \(\beta_{\mathcal F}=0\)}.  
 
If \(M\) is a curve  and \(\mathcal F\) is the foliation whose only leaf is \(M\)  then by Serre duality \(h^1 (M, K_M^2) =h^0 (M, TM)\), and we recover the fact that  every compact curve of higher genus has a projective structure. Notice, however, that this argument does not allow to conclude that rational and elliptic curves have such structures.

This discussion also shows that, if non empty, \emph{the moduli space of foliated projective structures on a given singular holomorphic foliation~$\mathcal{F}$ is an affine space directed by the space \(H^0(K_{\mathcal F}^2)\)}.

Despite Example~\ref{thm:kodfib}, it is quite common for a singular holomorphic foliation to carry a foliated projective structure. The following criteria is a consequence of Kodaira's vanishing theorem: 

\begin{lemma}  Let~$M$ be a compact  manifold of dimension~$n>1$, $\mathcal{F}$ a foliation by curves on~$M$ such that   \( K_{\mathcal F}^2\otimes K_M^*  \) is ample. Then, there exists a foliated projective structure on \( \mathcal F\).
\end{lemma}  

\begin{proof} Under the assumption,  $M$ is projective, and by Kodaira's vanishing theorem, \(H^1 ( K_{\mathcal F}^2) =0\), so~\(\beta_{\mathcal F}\) vanishes and the claim follows. \end{proof} 

\begin{rmk} This reasoning also shows that  if \(T_{\mathcal F}^{2}\) is ample,  we have a foliated projective structure: by Kodaira's  vanishing theorem,  	$h^{n-1}(T^2_\mathcal{F}\otimes K_M)=0$, and the latter, by Serre duality, equals~$h^1(K_M\otimes (T^2_\mathcal{F}\otimes K_M)^*)=h^1( K^2_\mathcal{F} )$. However, in this situation, we already know that there exists a foliated affine structure, so this is of little help.\end{rmk}

Let us illustrate the use of this lemma.  

\begin{prop} \label{p: foliations on products} Any singular holomorphic foliation on the product of a curve with the projective line carries a foliated projective structure. \end{prop}

\begin{proof} Let~\(C\) be a curve of genus \(g\),  $S=C\times \mathbf{P}^1$, and let \(\mathcal F\) be a foliation on~$S$. Curves of the form~$\{*\}\times \mathbf{P}^1$ will be called \emph{vertical}; those of the form~$C\times\{*\}$,  \emph{horizontal}. Let~$V\in H^2(M,\mathbf{Z})$ be the Poincar\'e dual of a vertical curve, $H\in H^2(M,\mathbf{Z})$ that of a horizontal one. These generate~$H^2(M,\mathbf{Z})$.  If~$\mathcal{F}$ is either the vertical or horizontal foliation, the proposition follows from the existence of projective structures on curves, so we will suppose that we are in neither case. Let us denote by \( n_h\) (resp. \( n_v\)) the number of tangencies of~$\mathcal{F}$ with a generic horizontal (resp. vertical) curve. We call \(n_h\) the horizontal degree and \(n_v\) the vertical one.  
	
We have that \( c_1(K_S ) = (2g-2) V - 2 H\) and that \( c_1(N_{\mathcal F} )= a H+b V \) for some~$a,b\in\mathbf{Z}$.  On a horizontal curve that is not invariant by~$\mathcal{F}$, a meromorphic section of \( N_{\mathcal F}^* \) induces a meromorphic one-form having \( n_h-(aH+bV) \cdot H \) zeros, so \(n_h-b= 2g-2\). The same reasoning shows that  \( n_v-a = -2\). To sum up, \[ c_1 (N_{\mathcal F}) = (n_v+2) H+ (n_h -2g+2) V. \] 
Since~$K_\mathcal{F}=K_M\otimes N_\mathcal{F}$ and~$K_M=\mathrm{pr}_1^*(K_C)\otimes \mathrm{pr}_2^*(K_{\mathbf{P}^1})$,  \( c_1 (K_{\mathcal F}) = n_v H + n_h V\), so
\begin{equation}\label{eq:lem_pos}c_1 (K_{\mathcal F}^2 \otimes K_S^*) = (2n_v+2) H + (2n_h - (2g-2))  V.\end{equation}
We now establish that the horizontal degree~$n_h$ is either zero  or  else \(n_h\geq 2g-2\).  The foliation \(\mathcal F\) is defined by a morphism \(i: T_\mathcal F \rightarrow T S\)  that vanishes on the singular set of~$\mathcal{F}$ (a finite number of points). Since \( TS= \mathrm{pr}_1^* (TC) \oplus \mathrm{pr}_2^* (T{\mathbf P}^1)\), the morphism~$i$ is given by  sections of \( K_{\mathcal F} \otimes\mathrm{pr}_1^* (T_C )\) and of \( K_{\mathcal F}\otimes \mathrm{pr}_2^* (T{\mathbf P}^1)\)  that vanish simultaneously on a finite set. Since the foliation is not the vertical one, the first section does not vanish identically. However, \( K_{\mathcal F} \otimes \mathrm{pr}_1^* (T C )= n_v H + (n_h-2g+2) V\), and such a section can only exist if both \(n_v\geq 0\) and \( n_h-2g+2\geq 0\),  proving the desired inequality. 

From this inequality and from~(\ref{eq:lem_pos}), \( K_{\mathcal F} ^2 \otimes K_S^*\) intersects positively~\( H\) and~\(V\), and hence every algebraic curve in~\(S\). This implies, by Nakai's criterion~\cite[Ch.~IV, Cor.~6.4]{bphv}, that it is ample. Hence, on a given singular holomorphic foliation on~\( C\times {\mathbf P}^1\), there exists a foliated projective structure.\end{proof}

\begin{rmk} It would be interesting to investigate the existence of foliated projective structures on general foliated ruled surfaces  (the work  of Gómez-Mont \cite{gomezmont}  seems a natural starting point).  Most foliations on these seem to have foliated affine structures. For instance, we leave to the reader the following: a more detailed inspection of the proof of Proposition~\ref{p: foliations on products}, together with the use of Lemma~\ref{l: criterion for affine foliated structures}, shows that, apart from  suspensions (vanishing vertical degree)  and eventually  foliations of horizontal degree~\(2g-2\) (the lower bound for the horizontal degree of a non-vertical foliation),  foliations on a product with a rational curve  carry foliated affine structures. We have not been able to decide whether the foliations of horizontal degree \(2g-2\) carry or not such structures.  \end{rmk}

\section{Local normal forms}\label{sec:local}	
At a non-singular point of a foliation, there are no local invariants neither  for affine nor for projective foliated structures. There are indeed local invariants at the singular points, beginning with  those of the foliation itself. The main results of this section, Theorems~\ref{thm:loc:aff} and~\ref{thm:loc:proj},  give local normal forms for generic foliated affine and projective structures on generic foliations. We prove that, in all dimensions, in the neighborhood of a generic  singular point of a foliation, a generic foliated projective structure is induced by an affine one, and that a generic foliated affine structure is given by a linear vector field having a constant Christoffel symbol. In particular, we prove that  the spaces of generic foliated affine and projective structures over a generic germ of singular foliation have both dimension one. We  also introduce the affine and projective \emph{ramification indices}, the main local invariants of foliated affine and projective structures at singular points of foliations, in terms of which the results of the following sections will be stated.

\subsection{Normal forms for foliated affine structures}

\subsubsection{Affine structures with singularities on curves} Let us begin by discussing some special singularities of affine structures on curves. Let~$U\subset\mathbf{C}$ be a neighborhood of~$0$, $U^*=U\setminus\{0\}$ and consider an affine structure on~$U^*$.  Let~$\alpha$ be the one-form in~$U^*$ measuring the difference from an auxiliary affine structure on~$U$ to the original one.  We  say that~$0$ is a \emph{singularity} for the affine structure if~$\alpha$ does not extend holomorphically to~$0$ (if~$\alpha$ extends holomorphically to~$0$, so does the affine structure), and that it is a \emph{Fuchsian} one if~$\alpha$ has a simple pole at~$0$. In this case, the residue of~$\alpha$ at~$0$ does not depend on the choice of the auxiliary affine structure on~$U$. The \emph{(normalized) affine angle} at~$0$ of the affine structure with singularities, $\measuredangle(0)\in\mathbf{C}$, is 
\[\measuredangle(0)=\mathrm{Res}(\alpha,p)+1.\]
The normalized affine angle of a non-singular point is~$1$. The normalized affine angle  of the affine structure with developing map~$z\mapsto z^\theta$ is~$\theta$; that of the one with developing map~$z\mapsto \log(z)$ vanishes. Following~\cite[Def.~4]{guillot-rebelo}, we define the (affine) \emph{ramification index} of a singular  affine structure as the reciprocal of the normalized affine angle.

We have a  classification of germs of singular affine structures with Fuchsian singularities  on curves, which may be attributed to Fuchs. It implies that, generically, the affine angle determines the singular affine structure.

\begin{prop}\label{aff1d} Consider an affine structure on a neighborhood of~$0$ in~$\mathbf{C}$ having a Fuchsian singularity at~$0$ with normalized affine angle~$\theta\in\mathbf{C}$. There exists a coordinate~$z$ around~$0$ where  the affine structure has as  developing map  
	\begin{itemize}
		\item $\log(z)$ if~$\theta=0$;
		\item $z\mapsto z^\theta$ if~$\theta\notin\mathbf{Z}^-$;
		\item either $z\mapsto z^\theta$ or $z\mapsto z^\theta+\log(z)$ if~$\theta\in\mathbf{Z}^-$.
	\end{itemize}
\end{prop}

\begin{proof} From the affine structure induced by a local coordinate~$z$, the difference with the singular affine structure has the form~$((\theta-1)/z+A(z))dz$ for some holomorphic function~$A$. The developing map  of the affine structure is thus a non-constant solution of~$zf''-((\theta-1)+zA)f'=0$.  The homogenized equation~$zf''-(\theta-1)f'$ has the solutions~$z^0$ and~$z^\theta$ (the original equation has \emph{indices}~$0$ and~$\theta$). According to Fuchs's theorem~\cite[\S 15.3]{ince}, if~$\theta$ is neither zero nor a negative integer, there is a solution of the form~$z^\theta h(z)$ with~$h(z)$ holomorphic and nonzero at~$0$. In this case, in the coordinate~$w=zh^{1/\theta}(z)$, the developing map is~$w^\theta$.  If~$\theta$ is zero or  a strictly negative integer, Fuchs's theorem affirms that there is a solution (in our setting, a developing map) of the form~$c\log(z)+z^\theta h(z)$, for some holomorphic function~$h$ taking the value~$1$ at~$0$  and some constant~$c$ (nonzero if~$\theta=0$). If~$c=0$ we are in a case identical to the previous one. Otherwise, if~$q(z)$ is such that~$q(0)=0$ and $e^{\theta q(z)}+cz^{-\theta} q(z)-h(z)=0$, then in the coordinate~$w=ze^{q(z)}$ the developing map reads~$c\log(w)+w^\theta$ (the existence of such a~$q$ follows from the Implicit Function Theorem). By conveniently scaling~$w$ and normalizing the developing map by post-composition by an affine map, we get the desired result. \end{proof}

\subsubsection{The foliated case} Let~$\mathcal{F}$ be a foliation tangent to a non-degenerate vector field defined on a neighborhood of~$0$ in~$\mathbf{C}^n$ and endowed with a foliated affine structure induced by the connection~$\nabla$. Let~$Z$ be a vector field tangent to~$\mathcal{F}$ and~$\gamma=\nabla(Z)$ its Christoffel symbol, as defined in~Section~\ref{ss: foliated affine structure}. Notice that~$\gamma$ extends holomorphically to~$0$. It follows from~(\ref{conn-fun}) that if~$\lambda_1,\ldots, \lambda_n$ are the eigenvalues of~$Z$ at~$0$, the ratio~$[\lambda_1:\cdots: \lambda_n:\gamma(0)]$ is an invariant of the foliated affine structure.

In dimension one, this invariant may be expressed in terms of the previously defined affine ramification index. Consider a singular affine structure on a neighborhood of~$0$ in~$\mathbf{C}$ given by the connection~$\nabla$. Let~$\gamma=\nabla(\lambda z\indel{z})$. The difference between the affine structure induced by the coordinate~$z$ and the first one is~$(\gamma(z)/\lambda-1)dz/z$  and thus, for the ramification index~$\nu$ of the original affine structure,   
\begin{equation}\label{angle-eigen-affine}\nu=\frac{\lambda}{\gamma(0)}.\end{equation}

In the foliated case  in~$(\mathbf{C}^n,0)$, if the eigenvalues at~$0$ of the vector field are~$\lambda_1$, \ldots, $\lambda_n$ and its Christoffel symbol~$\gamma$ does not vanish at~$0$ we will say that~$\nu_i=\lambda_i/\gamma(0)$ is a \emph{principal ramification index}.  From~(\ref{angle-eigen-affine}), 
\begin{equation}\label{eq:aff_ratio} [\lambda_1:\cdots:\lambda_n:\gamma(0)]=[\nu_1:\cdots:\nu_n:1].\end{equation}
In the generic non-degenerate case there will be~$n$ curves~$C_1,\ldots,C_n$ through~$0$, invariant by~$\mathcal{F}$, pairwise transverse, and tangent to the eigenspaces of the linear part of the vector field, and~$\nu_i$ will be the ramification index of the affine structure on~$C_i$ at~$0$.

Generically, the ratio~(\ref{eq:aff_ratio}) determines the foliated affine structure:

\begin{thm} \label{thm:loc:aff} Let~$\mathcal{F}$ be a foliation on a neighborhood of~$0$ in~$\mathbf{C}^n$, with a singularity at~$0$, tangent to a non-degenerate vector field~$Z$ satisfying Brjuno's condition~($\omega$). For a generic foliated connection~$\nabla$ on~$T_\mathcal{F}$, there exist coordinates where~$\mathcal{F}$ is tangent to a linear vector field~$Z'$  whose Christoffel symbol~$\nabla(Z')$ is constant. \end{thm}

We refer the reader to~\cite[Ch.~5]{arn} for details on Brjuno's condition~($\omega$), and only mention that it is satisfied by generic (in a measure-theoretic sense) linear parts.  The genericity of the affine structure  will be made precise further on. The proof of our theorem  will be an application of the following general result.

\begin{thm}[Brjuno, P\"oschel]  \label{thm:bbplus} Let~$Z=\sum_i \lambda_iz_i\indel{z_i}$ be a linear vector field on~$(\mathbf{C}^n,0)$.  Let~$F$ be a holomorphic function defined in the neighborhood of~$(0,0)$ in~$\mathbf{C}\times \mathbf{C}^n$ such that~$F(0,0)=0$ and consider the differential equation~$Zf=F(f,z)$ subject to the  condition~$f(0)=0$.  Let~$\mu=\partial F/\partial f|_{(0,0)}$ and suppose that~$\mu\neq\langle K,\lambda \rangle$ for every~$K\in(\mathbf{Z}_{\geq 0})^n$ with~$|K|\geq 2$. Let 
\[\omega'(m)=\min_{2\leq |K|\leq m} |\langle K,\lambda \rangle-\mu|.  \]
Then, if
\begin{equation}-\label{condbrjuno2}\sum_{\nu\geq 0} 2^{-\nu}\log\omega'(2^{\nu+1})<\infty,\end{equation}
the equation has a holomorphic solution (which is, moreover, unique). \end{thm}

In this theorem, the function~$f$ will be a solution of the differential equation if and only if the vector field~$Z\oplus F(\zeta,z)\indel{\zeta}$, defined in a neighborhood of the origin of~$\mathbf{C}^{n}\times\mathbf{C}$, has~$\zeta=f(z)$  as an invariant manifold. The condition~$\mu\neq\langle K,\lambda \rangle$ guarantees the existence of a formal solution, and~(\ref{condbrjuno2}) guarantees its convergence. For~$n=1$, the hypothesis on~$(\omega')$ is a superfluous one, and the result reduces to Briot and Bouquet's theorem~\cite[\S12.6]{ince}. Theorem~\ref{thm:bbplus} does not exactly appear in the literature in the above formulation. Brjuno's  announcement~\cite{brjuno_varieties} gives a similar statement, and we can find in~\cite{poschel} an  analogous result in the context of invariant manifolds for germs of diffeomorphisms; the proof of the latter may be adapted in a straightforward way to give a complete proof of the above theorem. (For the case where the~$\lambda_i$ belong to the Poincar\'e domain,  see  also~\cite{kaplan}, \cite{carrillo};  see  \cite[\S IX]{cha_stras} for an analogous result under Siegel-type Diophantine conditions.)

\begin{proof}[Proof of Theorem~\ref{thm:loc:aff}] Since~$Z$ satisfies Brjuno's condition~($\omega$), it is linearizable, so we may suppose that it is already linear. Suppose that~$f$ is a function such that
\begin{equation}\label{affinedo}Zf=\gamma(0)-f\gamma,\; f(0)=1.\end{equation}
The existence of such a function follows, generically, from Theorem~\ref{thm:bbplus}, which we may apply to equation~(\ref{affinedo}). In terms of the statement of Theorem~\ref{thm:bbplus}, $\mu=-\gamma(0)$; generically, $\mu\neq\sum_i m_i\lambda_i$, and  condition~(\ref{condbrjuno2}) is satisfied. The Christoffel symbol of the vector field~$Z'=fZ$, is, by construction, the constant~$\gamma(0)$. It remains constant in the coordinates where~$Z'$ is linear. \end{proof}
Notice that the condition~$-\gamma(0)\neq\sum_i m_i\lambda_i$ may be expressed solely in terms of the principal affine ramification indices.

\subsection{Normal forms for foliated projective structures}\label{sec:localproj}  
\subsubsection{Projective structures with singularities on curves}  Let~$U\subset\mathbf{C}$ be a neighborhood of~$0$, $U^*=U\setminus\{0\}$, and consider a projective structure on~$U^*$.  Let~$\beta$ be the  quadratic form in~$U^*$ measuring the difference from an auxiliary projective structure on~$U$ to this one. We  say that~$0$ is a \emph{singularity} for the  projective structure if~$\beta$ does not extend holomorphically to~$0$. A singularity of a   projective structure is said to be \emph{Fuchsian} if~$\beta$ has at most a double pole at~$0$. The quadratic residue~$\mathrm{Q}(\beta,0)$ of the quadratic form~$\beta$ at~$0$, $\mathrm{Q}((r/z^2+\ldots)dz^2,0)=r$,  does not depend on the choice of the auxiliary  projective structure. In this case, we define the \emph{(normalized) projective angle} at~$0$  of the  projective structure with singularities as~$\measuredangle(0)=\sqrt{1-2\mathrm{Q}(\beta,0)}$. It is only well-defined up to sign.  The normalized projective angle of the projective structure with developing map~$z\mapsto z^\theta$ is~$\pm \theta$. We define the \emph{projective ramification index} at~$0$  as the reciprocal of the normalized projective angle. Again, it is only well-defined up to sign.

We also have a local classification of projective structures with Fuchsian singularities in dimension one.

\begin{prop}\label{prop:locproj1d} Consider a projective structure on a neighborhood of~$0$ in~$\mathbf{C}$ having a Fuchsian singularity at~$0$ with normalized projective angle~$\theta\in\mathbf{C}$. Then, there exists a singular affine Fuchsian structure in its class. In particular, there exists a coordinate~$z$ around~$0$ where the developing map is given as in Proposition~\ref{aff1d}.
\end{prop}

\begin{proof} The difference from the projective structure induced by a local coordinate~$z$ to the singular one has the form~$S(z)dz^2$, $S(z)=\frac{1}{2}(1-\theta^2)z^{-2}+\ldots$. From~(\ref{eq:def-shw}), 
the affine structure with invariant~$g(z)dz$ is in the projective class of the original projective structure if~$g$   is a solution to the Riccati equation $g'=S+\frac{1}{2}g^2$ (if there is some~$f$ for which~$f''/f'=g$ and~$\{f,z\}=S$). For~$u=zg$, this equation reads  
\begin{equation}\label{eq:projtoaff}zu'=z^2S(z)+u+\frac{1}{2}u^2.\end{equation}
Let~$\theta$ be a root of~$\theta^2$ that is not a strictly positive integer. By the theorem of Briot and Bouquet~\cite[\S12.6]{ince},   equation~(\ref{eq:projtoaff}) has a holomorphic solution~$u(z)$ with~$u(0)=\theta-1$. The affine structure induced by~$gdz=udz/z$ is thus Fuchsian and induces the original projective structure.  \end{proof}

\subsubsection{The foliated case}\label{sec:norforproj} Let~$\mathcal{F}$ be a foliation tangent to a non-degenerate vector field defined in a neighborhood of~$0$ in~$\mathbf{C}^n$ and endowed with a foliated projective structure induced by the projective connection~$\Xi$. Let~$Z$ be a vector field tangent to~$\mathcal{F}$ and~$\rho=\Xi(Z)$ its Christoffel symbol, as defined in Section~\ref{ss: foliated projective structure}; notice that it extends holomorphically to~$0$. From~(\ref{christproj}), if~$\lambda_1,\ldots,\lambda_n$ are the eigenvalues of~$Z$ at~$0$, in the weighted projective space~$\mathbf{P}(1,\ldots,1,2)$, the  ratio~$[\lambda_1:\cdots:\lambda_n:\rho(0)]$ is an invariant of the foliated projective structure. Let us relate this invariant, in dimension one, to the previously defined projective ramification index. Consider a singular projective structure on~$(\mathbf{C},0)$ and let~$\rho=\Xi(\lambda z\indel{z})$. The difference of the  projective structure with coordinate~$z$ and the singular one  is~$\frac{1}{2}(1+2\rho/\lambda^2)dz^2/z^2$, and thus, for the  projective ramification index~$\nu$, $\nu^2=-\frac{1}{2}\lambda^2/\rho(0)$. In particular, 
\begin{equation}\label{fromeigentoindex} [\lambda_1:\cdots:\lambda_n:-2\rho(0)]=[\nu_1:\cdots:\nu_n:1] \text{ in }   \mathbf{P}(1,\ldots,1,2).\end{equation}	
In this setting, the numbers~$\nu_i$ are said to be the \emph{principal projective ramification  indices} of the foliated projective structure at~$0$.

\begin{rmk}\label{rem:even} The individual principal projective ramification indices are only well-defined up to sign. More generally, only their even functions  are well-defined.
\end{rmk}

\begin{rmk}\label{rem:indaffproj}If a foliated affine structure is considered as a projective one, its affine and projective ramification indices coincide (within the limitations given by the previous remark).
\end{rmk}

The analogue of Theorem~\ref{thm:loc:aff} for foliated projective structures is the following one:

\begin{thm} \label{thm:loc:proj} Let~$\mathcal{F}$ be a foliation on a neighborhood of~$0$ in~$\mathbf{C}^n$, with a singularity at~$0$, generated by a non-degenerate vector field~$Z$ satisfying Brjuno's condition~($\omega$). For a generic foliated projective structure on~$\mathcal{F}$, 
\begin{itemize}
	\item	there exists a foliated affine structure in its  class, and
	\item there exist coordinates where~$\mathcal{F}$ is tangent to a linear vector field  having a constant Christoffel symbol. 
\end{itemize} 
\end{thm}
\begin{proof}  Let~$\mathcal{F}$ be a foliation endowed with a foliated projective structure with connection~$\Xi$. Let~$Z$ be a vector field tangent to~$\mathcal{F}$, and suppose that it is linear. Let~$\rho=\Xi(Z)$ be  its  Christoffel symbol. From  formula~\eqref{christproj}, if~$\gamma$ is a function such that~$Z\gamma=\frac{1}{2}\gamma^2+\rho$, there exists, like in Proposition~\ref{prop:locproj1d}, a foliated  affine structure inducing the given projective one, with connection~$\nabla$, such that~$\nabla(Z)=\gamma$.  We may resort to Theorem~\ref{thm:bbplus} to establish the  existence of a solution to this equation with one of the initial conditions~$\gamma(0)$ such that ~$\gamma^2(0)+2\rho(0)=0$. For the hypothesis of the theorem, $\mu=\gamma(0)$, and according to it, we have  solutions to the equation whenever~$\gamma(0)\neq\sum_i m_i\lambda_i$ and condition~(\ref{condbrjuno2}) is satisfied.  Theorem~\ref{thm:loc:aff}  establishes the second part of our claim.\end{proof}

\section{An index theorem  for foliated affine  structures} \label{sec:affindex}

The existence of an affine structure on a curve imposes topological restrictions on it, and the only compact curves admitting them are elliptic ones. Similarly, the existence of a foliated affine structure imposes topological restrictions on both the foliation and the ambient manifold, and conditions the local behavior of the foliated affine structure at its singular points.

We will make this precise through an index theorem relating the affine  ramification indices defined in the previous section with some topological data depending only on the foliation.  Some index theorems of the like follow directly from the Baum and Bott's one~\cite{baum-bott} since, generically, from~(\ref{eq:aff_ratio}), the ratios of the eigenvalues of a vector field tangent to a foliation  at a singular point (in terms of which the Baum-Bott index theorem is expressed in many situations) are the ratios of the principal affine ramification indices. We are nevertheless interested in results that truly depend on the foliated affine structure and not just on the foliation that supports it: 
\begin{thm}\label{ourdescartes} Let~$M$ be a compact complex manifold of dimension~$n$, $\mathcal{F}$  a holomorphic foliation by curves on~$M$ having only isolated non-degenerate singularities~$p_1,\ldots, p_k$. Consider a foliated  affine structure subordinate to~$\mathcal{F}$ having at each one of the singularities a non-vanishing Christoffel symbol, and let~$\nu_1^{(i)}$, \ldots, $\nu_n^{(i)}$ be the principal affine ramification indices at~$p_i$. Then,
\[ (-1)^n\sum_{i=1}^k\frac{1}{\nu_1^{(i)} \cdots \nu_n^{(i)}}= c_1^n(T_\mathcal{F}). \]
\end{thm}	
Here, $c_1(T_\mathcal{F})\in H^2(M,\mathbf{Z})$ is the first Chern class of~$T_\mathcal{F}$, $c_1^n(T_\mathcal{F})\in H^{2n}(M,\mathbf{Z})$, and, as usual, we have identified cohomology classes of top degree with their evaluation on the fundamental class.  For instance, if~$\mathcal{F}$ is a foliation of degree~$d$ on~$\mathbf{P}^n$, $c_1^n(T_\mathcal{F})=(1-d)^n$. In the case where~$M$ is a curve ($n=1$), there are no singularities ($k=0$) and~$T_\mathcal{F}=TM$: the result reduces to~$c_1(TM)=0$.

A foliated affine  structure along the foliation~$\mathcal{F}$ defines naturally a \emph{geodesic vector field} on~$T_\mathcal{F}$.  Theorem~\ref{ourdescartes} will follow from applying Lehmann's index theorem~\cite{lehmann} to the foliation induced by this vector field relative to the zero section.

\subsection{The geodesic vector field}  

Consider a foliated affine  structure along the foliation~$\mathcal{F}$.
For every~$v\in T_\mathcal{F}$ such that~$\pi(v)$ is not a singular point of~$\mathcal{F}$, there is a  geodesic of the associated connection, $c:(U,0)\to (M,\pi(v))$,  tangent to~$\mathcal{F}$, such that~$c'(0)=v$ (the germ of~$c$ at~$0$ is unique). The derivative gives a lift~$\widetilde{c}:U\to T_\mathcal{F}$ with~$\pi(\widetilde{c}(t))=c(t)$  and~$\widetilde{c}(t)=c'(t)$. The vector field on~$T_\mathcal{F}\setminus\pi^{-1}(\mathrm{Sing}(\mathcal{F}))$ that has this curve as its integral one through~$v$ extends, by Hartog's theorem, to all of~$T_\mathcal{F}$. This is the \emph{geodesic vector field} of the foliated affine structure.

Local expressions may be given as follows.
Let~$\{U_i\}_{i\in I}$ be a cover of~$M$ by open subsets such that, in~$U_j$, $\mathcal{F}$ is given by the vector field~$Z_j$. If~$U_i\cap U_j\neq\emptyset$,  let~$g_{ij}:U_i\cap U_j\to\mathbf{C}^*$ be the function such that~$Z_i=g_{ij}Z_j$. The line bundle~$T_\mathcal{F}$ is obtained by gluing the sets in~$\{U_i\times \mathbf{C}\}_{i\in I}$ by means of the identification
\begin{equation}\label{gluetangent}(u,\zeta_j)=(u,g_{ij}\zeta_j)\end{equation} if~$U_i\cap U_j\neq\emptyset$. 

Let now the foliated affine structure come into play. Let~$\gamma_j:U_j\to\mathbf{C}$ be the Christoffel symbol~$\nabla(Z_j)$. Consider, in~$U_j\times \mathbf{C}$, the vector field 
\begin{equation}\label{loc-vf} X_j=\zeta_j Z_j- \gamma_j\zeta_j^2\del{\zeta_j}.\end{equation}
In~$(U_i\times\mathbf{C})\cap (U_j\times\mathbf{C})$, under~(\ref{gluetangent}), this vector field reads~$g_{ij}\zeta_i Z_j- (g_{ij}\gamma_j+Z_jg_{ij})\zeta_i^2\indel{\zeta_i}$, which, by Leibniz's rule~(\ref{eq: Leibniz}), equals~$X_i$. This shows that~(\ref{loc-vf}) defines a global holomorphic vector field~$X$  on the total space of~$T_\mathcal{F}$. We will establish that this is the geodesic vector field of~$\nabla$.

The vector field~$H$ on~$T_\mathcal{F}$ given by~$\zeta_j\indel{\zeta_j}$ in $U_j\times\mathbf{C}$ is globally well-defined. We have the   relation~$[H,X]=X$. In its integral form, it implies that if~$(z(t),\zeta_j(t))$ is a solution to~$X_j$ then~$(z(at+b),a\zeta_j(at+b))$ is also a solution. Since all the solutions above a given point may be constructed in this way, the vector field~$X$ gives a class of parametrizations of the leaves of~$\mathcal{F}$ that is invariant under precompositions by affine maps. The inverses of these parametrizations  form the atlas of charts of a foliated affine structure.

Let us prove that the foliated affine structure  associated to~$X$ is exactly the one we started with, that the parametrized solutions of~$X$ project onto the geodesics of our original foliated affine structure. If the vector field~$Z_j$ is such that~$\nabla(Z_j)$ vanishes identically then, on the one hand, the geodesics of~$\nabla$ are the integral curves of~$Z_j$ (with their natural parametrization) and its constant multiples; on the other, $X_j$ reduces  to~$\zeta_jZ_j$  ($\zeta_j$ is a first integral), and the integral curves of the latter project also onto the integral curves of the constant multiples of~$Z_j$. We conclude that~$X$ is the geodesic vector field of~$\nabla$.

\begin{rmk} This gives yet another definition of a foliated affine structure: \emph{a vector field~$X$ on~$T_\mathcal{F}$ projecting onto~$\mathcal{F}$ such that~$[H,X]=X$.}  The projections of the solutions of such a vector field induce an affine structure along the leaves of~$\mathcal{F}$. Further, it has local expressions of the form~(\ref{loc-vf}), and the connection~$\nabla:T_\mathcal{F}\to \mathcal{O}(M)$ locally defined by~$\nabla(Z_j)=\gamma_j$ gives a globally well-defined connection.\end{rmk}

\begin{rmk}
The geodesic vector field~$X$ is a quasihomogeneous one, for~$[H,X]=X$.  The singularities of~$\mathcal{F}$ are in correspondence with the fibers of~$T_\mathcal{F}$ along which~$H\wedge X=0$; those  with non-vanishing  Christoffel symbols correspond to the fibers where~$X$ does not vanish identically. Let~$p\in\mathrm{Sing}(\mathcal{F})$ be one of these, and let~$\nu_1,\ldots,\nu_n$ be its principal ramification indices. For the solution of~$X$ contained in the fiber above~$p$, one can define its~\emph{Kowalevsky exponents}, complex  numbers that localize some integrability properties of~$X$. In our case, with the normalizations found in~\cite{goriely}, these are~$-1, \nu_1,\ldots,\nu_n$.\end{rmk}

The vector field~$X$ induces a foliation by curves~$\mathcal{G}$ on~$T_\mathcal{F}$  that  leaves the zero section invariant.  This is exactly the setting of Lehmann's theorem.

\subsection{Lehmann's theorem}

For the proof of Theorem~\ref{ourdescartes}, we will use an index theorem due to Lehmann  which generalizes the Camacho-Sad index theorem to higher dimensions~\cite{lehmann}. Let us recall it in the generality that will suit our needs. We follow the normalizations  and sign conventions found in~\cite{suwa}.

Let~$V$ be a manifold of dimension~$n+1$, $M\subset V$ a codimension one smooth compact submanifold with normal bundle~$N_M$, and~$\mathcal{G}$ a foliation by curves on~$V$ leaving~$M$ invariant. Suppose that the singularities of the foliation induced by~$\mathcal{G}$ on~$M$ are isolated. For such a singularity~$p$, in coordinates~$(z_1,\ldots,z_n,w)$ centered at~$p$, where~$M$ is given by~$w=0$ and~$\mathcal{G}$ is induced by~$X=\sum_{i=1}^n a_i(z,\zeta)\indel{z_i}+w b(z,w) \indel{w}$, 
define 
\[ \mathrm{Res}_\mathcal{G}(c_1^n,M, p)=(-1)^n\left(\frac{i}{2\pi}\right)^n\int_T \frac{b^n(z,0)}{\prod_{i=1}^n a_i(z,0)}dz_1\wedge\cdots\wedge dz_n,\]
with~$T=\{w=0\}\cap(\cap_{i=1}^n\{\|a_i(z,0)\|=\epsilon\})$ for some sufficiently small~$\epsilon$. (This number is well-defined.)  Lehmann's theorem affirms that
\[ \sum_{p\in\mathrm{Sing}(\mathcal{G}|_M)}\mathrm{Res}_\mathcal{G}(c_1^n,M,p)=c_1^n(N_M),\]
where~$c_1(N_M)\in H^2(M,\mathbf{Z})$ denotes the first Chern class of~$N_M$.

If the restriction of~$X$ to~$w=0$ is non-degenerate at~$p$ and the eigenvalues of the linear part of this restriction are $\lambda_1,\ldots,\lambda_n$, we have that~$\mathrm{Res}_\mathcal{G}(c_1^n,M, p)=b^n(0)(\lambda_1\cdots\lambda_n)^{-1}$.

\begin{proof}[Proof of Theorem~\ref{ourdescartes}] 
Let~$\mathcal{G}$ be the foliation on~$T_\mathcal{F}$ induced by the geodesic vector field~$X$.	It leaves the zero section~$M$ invariant. If~$\mathcal{F}$ is generated by~$Z=\sum_i a_i(z)\indel{z_i}$  in a neighborhood of~$p$, $\mathcal{G}$ is, in a neighborhood of~$p$ in~$T_\mathcal{F}$, tangent to the vector field~$\sum_{i=1}^n a_i(z)\indel{z_i}-\zeta\gamma(z)\indel{\zeta}$. If~$\gamma(0)\neq 0$ and $\nu_1$, \ldots, $\nu_n$ are the principal affine ramification indices of the foliated affine structure at~$p$, then,  by~(\ref{angle-eigen-affine}), 
\[\mathrm{Res}_\mathcal{G}(c_1^n,M, p)=\frac{(-\gamma(0))^n}{\lambda_1\cdots\lambda_n}=\frac{(-1)^n}{\nu_1   \cdots  \nu_n }.\]
On the other hand, by construction, $N_M$ is  exactly~$T_\mathcal{F}$. A straightforward application of Lehmann's theorem yields Theorem~\ref{ourdescartes}.\end{proof}

\section{An index theorem for foliated   projective structures} \label{sec:projindex}

Every compact curve admits a projective structure but, as we have seen, not every foliation in a surface admits a foliated one. When foliated projective structures do exist, the  foliation and the ambient manifold impose  conditions on the  behavior of the structure at the singular points of the foliation. The results in this section will cast these in a precise form.
	
The Baum and Bott index theorem~\cite{baum-bott} will be behind the formulation and the proof of our result, and we begin by recalling some of the notions and terms that appear in its statement.  Let~$\varphi(x_1,\ldots,x_k)$ be a symmetric homogeneous polynomial with complex coefficients. Define the polynomial~$\widetilde{\varphi}$ through the equality~$\widetilde{\varphi}(\sigma_1,\ldots,\sigma_k)=\varphi(x_1,\ldots,x_k)$, where~$\sigma_i  = \sum _{j_1 < \ldots <j_i} x_{j_1} \ldots x_{j_i} $ is the $i$-th elementary symmetric polynomial in~$x_1,\ldots,x_k$. For a vector bundle~$V$, let~$c_i(V)$ denote the $i$-th Chern class of~$V$, $c(V)$ its total Chern class, and let $\varphi(c(V))\in H^{2k}(N,\mathbf{Z})$ be given by~$\widetilde{\varphi}(c_1(V),\ldots,c_k(V))$.  This definition extends to the context of virtual vector bundles, which are elements of the \(K\)-theory of \(M\). We refer the reader to \cite{bott-kx} for facts around virtual vector bundles and their Chern classes.

Let~$\varphi(x_1,\ldots,x_{n+1})$ be a symmetric homogeneous polynomial of degree~$n+1$. We will distinguish the variable~$x_{n+1}$.  For~$i=0,\ldots, n+1$, define the symmetric homogeneous polynomial of degree~$i$ in~$n$ variables~$\widehat{\varphi}_i$ through the equality
\begin{equation}\label{symdecomp}\varphi(x_1,\ldots,x_n,x_{n+1})=\sum_{i=0}^{n+1}x_{n+1}^{n+1-i}\widehat{\varphi}_{i}(x_1,\ldots,x_n).\end{equation}
In particular, for the odd part (with respect to~$x_{n+1}$)~$\varphi_\mathrm{odd}$ of~$\varphi$,
\[\varphi_\mathrm{odd}(x_1,\ldots,x_n,x_{n+1})=\sum_{j=0}^{\lfloor n/2\rfloor}x_{n+1}^{2j+1}\widehat{\varphi}_{n-2j}(x_1,\ldots,x_n).\]

\begin{thm}\label{ourindthmproj}  
Let~$M$ be a compact complex manifold of dimension~$n$, $\mathcal{F}$  a holomorphic foliation by curves on~$M$ having only isolated non-degenerate singularities~$p_1,\ldots, p_k$. Consider a holomorphic foliated projective structure subordinate to~$\mathcal{F}$ for which the Christoffel symbols do not vanish at the singularities, and let~$\nu_1^{(i)}$, \ldots, $\nu_n^{(i)}$ be the  principal projective ramification indices at~$p_i$. Let~$\varphi(x_1,\ldots,x_{n+1})$ be a symmetric homogeneous polynomial of degree~$n+1$. Then, with the previous notations,	\begin{equation}\label{eq:tracepown}
	\sum_{i=1}^k    \frac{\varphi_\mathrm{odd}(\nu_1^{(i)},\cdots, \nu_n^{(i)},1)}{\nu_1^{(i)}  \cdots  \nu_n^{(i)}} = \sum_{j=0}^{\lfloor n/2\rfloor}  c_1^{2j}(T_\mathcal{F}) \widehat{\varphi}_{n-2j}(c(TM-T_\mathcal{F})).   \end{equation} 
\end{thm}
The summands in the left-hand side of~(\ref{eq:tracepown}) are, in agreement with Remark~\ref{rem:even}, well-defined.

\begin{exam}\label{thmexasimplest} If~$n$ is even, for~$\varphi=\sum_{i=1}^{n+1} x_i^{n+1}$, $ \varphi_\mathrm{odd}(x_1,\ldots,x_{n+1})=x_{n+1}^{n+1}$,  and~(\ref{eq:tracepown}) becomes simply formula~(\ref{for:projeven}). 
\end{exam}

\begin{exam}\label{thmexase}	
If~$n$ is odd, $n=2m+1$, for~$\varphi=\sum_{i\neq j}  x_i^{n}x_j$, $ \varphi_\mathrm{odd}=x_{n+1}^n\sum_{i=1}^nx_{i}+x_{n+1}\psi(x_1,\ldots,x_n)$  for~$\psi(x_1,\ldots,x_n)=\sum_{i=1}^nx_{i}^n$,  and~(\ref{eq:tracepown}) becomes 
\begin{multline*}	\sum_{i=1}^k    \frac{\nu_1^{(i)} + \cdots+  \nu_n^{(i)}}{\nu_1^{(i)}  \cdots  \nu_n^{(i)}} +\sum_{i=1}^k    \frac{\left(\nu_1^{(i)}\right)^n + \cdots+  \left(\nu_n^{(i)}\right)^n}{\nu_1^{(i)}  \cdots  \nu_n^{(i)}} \\ =   c_1^{2m}(T_\mathcal{F})c_1(TM-T_\mathcal{F}) + \widetilde{\psi}(TM-T_\mathcal{F}),  \end{multline*}
where~$\widetilde{\psi}$ is  as in the Baum-Bott index theorem. The latter implies that the second summands in each side are equal, and the equality reduces to~(\ref{for:projodd}). \end{exam}

In particular, in even dimensions, with Remark~\ref{rem:indaffproj} taken into account, Theorem~\ref{ourindthmproj} extends Theorem~\ref{ourdescartes} to the projective setting. (In odd dimensions,  Theorem~\ref{ourdescartes} is exclusively affine since, to begin with, from Remark~\ref{rem:even}, its statement does not make sense in the projective case.)

\begin{exam} Let~$\mathcal{F}$ be the foliation on~$\mathbf{P}^2$ given by the pencil of  conics through four points in general position. There are three singular conics, pairs of lines associated to the three ways in which the four points in the base can be taken in pairs; these degenerations correspond  to the ``pinching'' of a   loop.  Beyond the four base points of the pencil, $\mathcal{F}$ has three other singular points, those where  the lines in each pair intersect (all of them are non-degenerate). Fix a foliated projective structure on~$\mathcal{F}$. For a non-singular conic in the pencil, the base points give four Fuchsian singularities of the projective structure (as defined in Section~\ref{sec:localproj}), with ramification indices~$\nu_1$, \ldots, $\nu_4$, which are independent of the conic. In a degenerate conic in the pencil the pinched loop produces, in each one of its lines,  a Fuchsian singular point of the projective structure  having a ramification index~$\mu_i$,  independent of the line. Theorem~\ref{ourindthmproj} in the instance of Example~\ref{thmexasimplest} gives that
\[ \frac{1}{\nu_1^2}+\frac{1}{\nu_2^2}+\frac{1}{\nu_3^2}+\frac{1}{\nu_4^2}-\frac{1}{\mu_1^2}-\frac{1}{\mu_2^2}-\frac{1}{\mu_3^2}=1,\]
and, in particular, that the three degenerations are not independent. 
\end{exam}

\begin{rmk} For a foliation  of degree two on~$\mathbf{P}^2$ its canonical bundle is~\(O(1)\), its square~\(O(2)\), and the dimension of the affine space of foliated projective structures is six. For such a  foliation~$\mathcal{F}$,   the rational map that to each foliated projective structure subordinate to~$\mathcal{F}$ associates the seven expressions~$\nu_1^{(i)}\nu_2^{(i)}$ takes values in the six-dimensional variety given by~(\ref{for:projeven}). A calculation we have made with a computer algebra system  shows that,  at a generic foliated projective structure subordinate to the foliation of the previous example, the differential of this map has full rank. In particular, together with Corollary~3.3 in~\cite{guillot-quadratic}, this implies that \emph{in the space of foliations of degree two of~$\mathbf{P}^2$ endowed with a foliated projective structure, the projective ramification indices determine a generic element (both the foliation and the projective structure) up to a finite indeterminacy}. 
\end{rmk}

\begin{rmk}  In Theorem~\ref{ourindthmproj}, singularities of~$\mathcal{F}$ with vanishing Christoffel symbols may be considered (still under  the hypothesis of non-degeneracy). The contribution to the left-hand side of~(\ref{eq:tracepown}) of such a singularity reduces to its Baum-Bott index associated to~$\widehat{\varphi}_n$. Details will be left to the reader.
\end{rmk}

For the proof of Theorem~\ref{ourindthmproj}, we will construct a geodesic vector field for the foliated  projective structure and apply the  Baum-Bott index theorem to a foliation associated to it.

\subsection{The geodesic vector field and its projectivization}

Let~$\mathcal{F}$ be a foliation by curves on the manifold~$M$. Let us recall the definition of the vector bundle $J^1T_\mathcal{F}\to M$ of one-jets of~$T_\mathcal{F}$. Let~$p\in M$ and let~$Z$ be an auxiliary vector field generating~$\mathcal{F}$ in a neighborhood of~$p$. Two vector fields tangent to~$\mathcal{F}$, $fZ$ and~$gZ$,  are said to have the same $1$\nobreakdash-jet at~$p$ if~$f(p)=g(p)$ and~$(Zf)(p)=(Zg)(p)$. This   depends only on~$\mathcal{F}$ and not on the auxiliary vector field~$Z$. The union  of all $1$-jets of vector fields tangent to~$\mathcal{F}$ over all points of~$M$ has a natural rank-two vector bundle structure, that we will denote by~$J^1T_\mathcal{F}\to M$. There is a natural linear projection~$j_0:J^1T_\mathcal{F}\to T_\mathcal{F}$ that maps each $1$-jet to its associated~$0$-jet. 

If~$\mathcal{F}$ is endowed with a foliated projective structure  and~$p\in M\setminus\mathrm{Sing}(\mathcal{F})$,  a \emph{geodesic} through~$p$ is a parametrized curve~$f:(U,0)\to (M,p)$,  $U\subset \mathbf{C}$, $0\in U$, which is tangent to~$\mathcal{F}$ and which induces the given projective structure in the leaf of~$\mathcal{F}$ through~$p$ (which is the inverse of a projective chart).  For the tautological projective structure in~$\mathbf{P}^1$, in the affine chart~$[z:1]$, the geodesics through~$0$ are those of the form~$t\mapsto at/(1-bt)$ with~$a\in\mathbf{C}^*$, $b\in\mathbf{C}$. Their corresponding  velocity vector fields are~$a^{-1}(a+bz)^2\indel{z}$. Each one of them is characterized by its~$1$-jet  at~$0$. Furthermore, every $1$-jet of vector field   with a non-vanishing $0$-jet is realized by the velocity vector field of a geodesic.

In this way, above the regular part of~$\mathcal{F}$,  the geodesics of a foliated projective structure lift into~$J^1 T_\mathcal{F} \setminus\ker(j_0)$ through their velocity vector fields, and there is a lift of a unique geodesic through every point in~$J^1 T_\mathcal{F}\setminus\ker(j_0)$. There is thus a natural  vector field on~$J^1 T_\mathcal{F} \setminus\ker(j_0)$ associated to a foliated projective structure.  This is its \emph{geodesic vector field}.

If~$f:U\to M$ is a geodesic defined in a neighborhood of~$t=0$ and~$\left(\begin{array}{cc}a & b \\ c & d \end{array}\right)\in\mathrm{SL}(2,\mathbf{C})$ is sufficiently close to the identity, $\displaystyle f\left(\frac{at+b}{ct+d}\right)$ is also a geodesic: we have a local action of~$\mathrm{SL}(2,\mathbf{C})$ on~$J^1 T_\mathcal{F} \setminus\ker(j_0)$ induced by the foliated projective structure.

In coordinates,   this looks as follows.
Let~$\{U_i\}_{i\in I}$ be a cover of~$M$ by open subsets such that $\mathcal{F}$ is given by the holomorphic vector field~$Z_j$ in~$U_j$.  Let~$g_{ij}:U_i\cap U_j\to\mathbf{C}^*$ be  such that~$Z_i=g_{ij}Z_j$. Consider the map~$h_i:J^1 T_\mathcal{F}|_{U_i}\to\mathcal{O}(U_i)\times \mathcal{O}(U_i)$ given by: 
\begin{equation}\label{twocomponentsjets}fZ_i\mapsto (f,Z_i(f)).\end{equation}
It is well-defined and linear, and gives a trivialization of~$J^1 T_\mathcal{F} $ over~$U_i$. For
\begin{equation}\label{cocycleproj}\psi_{ji}=\left(\begin{array}{cc} g_{ij} & 0 \\   Z_j(g_{ij}) & 1\end{array} \right),\end{equation}
we have that~$\psi_{ji}h_i\equiv h_j$:
\begin{multline*}\psi_{ji}h_i(fZ_i)=\left(\begin{array}{cc} g_{ij} & 0 \\   Z_j(g_{ij}) & 1\end{array} \right)\left(\begin{array}{c} f  \\    Z_i(f)\end{array} \right)=\left(\begin{array}{c} g_{ij}f  \\  fZ_j(g_{ij})+  Z_i(f)\end{array} \right) \\ =\left(\begin{array}{c} g_{ij}f  \\ fZ_j(g_{ij})+ g_{ij} Z_j(f)\end{array} \right) =\left(\begin{array}{c} g_{ij}f  \\  Z_j(g_{ij} f)\end{array} \right)=h_j(fg_{ij}Z_j)=h_j(fZ_i). \end{multline*}
In particular, the cocycle conditions~$\psi_{ij}\psi_{jk}\psi_{ki}\equiv \mathbf{I}$ and $\psi_{ij}\psi_{ji}\equiv \mathbf{I}$ hold.

The rank two vector bundle~$J^1 T_\mathcal{F} $  may in consequence be obtained by gluing the sets in~$\{U_i\times\mathbf{C}^2\}_{i\in I}$ by means of the identifications
\begin{equation}\label{for:ranktwo}\left(\begin{array}{c} \zeta_j \\ \xi_j  \end{array}\right)
=\psi_{ji}\left(\begin{array}{c} \zeta_i \\ \xi_i  \end{array}\right).\end{equation}
In these coordinates, $j_0(\zeta_i,\xi_i)= \zeta_i$. From~(\ref{cocycleproj}),  the line subbundle~$\ker(j_0)$ is trivial.

Let us make explicit the previously defined local action of~$\mathrm{SL}(2,\mathbf{C})$ on $J^1 T_\mathcal{F}$ through the action of the one-parameter subgroups of the standard basis of~$\mathfrak{sl}(2,\mathbf{C})$. The vector fields on~$J^1T_\mathcal{F}$ associated to these will satisfy the same Lie-algebraic relations than their infinitesimal counterparts do.

Let~$f:U\to M$ be a geodesic.  That the action via reparametrization of the one-parameter subgroup~$\left(\begin{array}{cc} 1 & s \\ 0 & 1 \end{array}\right)$ preserves  geodesics is equivalent to the fact that if~$f(t)$ is a parameterized geodesic, so is~$f(t+s)$ for every (sufficiently small) fixed~$s$. This reparametrization comes from the flow of the geodesic vector field. The one induced by the one-parameter subgroup~$\left(\begin{array}{cc}e^{s/2} & 0 \\ 0 & e^{-s/2} \end{array}\right)$ yields the   geodesic~$t\mapsto f(e^st)$. It multiplies the velocity vector field of~$f$ by the factor~$e^s$. In coordinates, it multiplies each one of the two components of~(\ref{twocomponentsjets}) by~$e^s$, and is thus induced by the vector field~$H$ on~$J^1 T_\mathcal{F} $ that, in~$U_j\times\mathbf{C}^2$, reads~$\zeta_j\indel{\zeta_j}\oplus \xi_j\indel{\xi_j}$ (it retains its expression in the other charts and is globally well-defined). Let us now consider the more interesting case of the one-parameter subgroup~$\left(\begin{array}{rr} 1 & 0 \\ -s & 1 \end{array}\right)$. Let us suppose that we are on a curve where we have a local coordinate~$z$ and that~$f(0)=0$.  The reparametrization gives the geodesic~$\displaystyle f\left(\frac{t}{1-st}\right)$, whose velocity vector field is
\[(1+sf^{-1}(z))f'(f^{-1}(z))\del{z}.\]
With respect to the vector field~$Z=g(z)\indel{z}$, the~$1$-jet of its velocity vector field,  as defined in~(\ref{twocomponentsjets}), is
\[\left(\frac{f(0)}{g(0)},-\frac{g'(0)}{g(0)}+\frac{f''(0)}{f'(0)}+2s\right).\]
The action of this one-parameter subgroup induces the vector field~$Y$ on~$J^1 T_\mathcal{F} $ that reads~$2\indel{\xi_j}$
on~$U_j\times\mathbf{C}^2$. As expected,  $[H,Y]=-Y$.

If~$X$  is a vector field on~$J^1 T_\mathcal{F}$ giving  the geodesic vector field of  a foliated projective structure, its integral curves project to curves of~$\mathcal{F}$ and, together with the vector fields~$H$ and~$Y$,  it satisfies the~$\mathfrak{sl}(2,\mathbf{C})$ relations
\begin{equation}\label{sl2}[Y,X]=2H, [H,X]=X, [H,Y]=-Y.\end{equation} 
The integral form of these relations gives the reparametrization of the geodesics on~$J^1 T_\mathcal{F}$: if~$(z(t),\zeta_j(t),\xi_j(t))$ is a solution of~$X_j$ and~$\left(\begin{array}{cc} a & b \\ c & d \end{array}\right)\in\mathrm{SL}(2,\mathbf{C})$ is sufficiently close to the identity, 
\[\left(z\left(\frac{at+b}{ct+d}\right), \frac{1}{(ct+d)^2}\zeta_j\left(\frac{at+b}{ct+d}\right), \frac{1}{(ct+d)^2}\xi_j\left(\frac{at+b}{ct+d}\right)-\frac{2c}{ct+d}\right)\]
is also a solution of~$X$.

\begin{rmk}
The above formula is  similar to the ``invariance condition'' enjoyed by Halphen's system~\cite{halphen}. This is not surprising, since the latter is essentially a compactification of the geodesic flow of the projective structure on the thrice-puncture sphere given by uniformization; see~\cite[Section~3]{guillot-halphen}.
\end{rmk}

Let us explicitly construct the vector field~$X$ associated to a  foliated projective connection~$\Xi:T_\mathcal{F}\to\mathcal{O}_M$. Let~$\rho_i:U_i\to\mathbf{C}$ be the Christoffel symbol~$\Xi(Z_i)$ of~$Z_i$. Consider, on~$U_j\times \mathbf{C}^2$, the vector field 
\[ X_j=\zeta_j Z_j\oplus \zeta_j\xi_j\del{\zeta_j}\oplus\left(\frac{1}{2}\xi_j^2-\rho_j\zeta_j^2\right)\del{\xi_j}.\]
In~$(U_i\times\mathbf{C}^2)\cap (U_j\times\mathbf{C}^2)$, under~(\ref{for:ranktwo}),
$X_j$ reads
\[g_{ij}\zeta_i Z_j\oplus \zeta_i\xi_i\del{\zeta_i}\oplus\left(\frac{1}{2}\xi_i^2-\left[g_{ij}^2 \rho_j+g_{ij}Z_j^2 g_{ij} -\frac{1}{2}(Z_jg_{ij})^2  \right]\zeta_i^2\right)\del{\xi_i},\]
which, by the Leibniz rule~(\ref{christproj}), is exactly~$X_i$. Thus, these vector fields glue into a globally-defined holomorphic vector field~$X$ on~$J^1 T_\mathcal{F}$. It satisfies the relations~(\ref{sl2}); since the projections of its integral curves onto~$M$ differ by precompositions with fractional linear transformations, it induces a foliated projective structure.

Let us identify the foliated projective structure induced by the  vector field~$X$ just defined. Let us do so in dimension one, in a coordinate~$z$ where~$Z$ is~$\indel{z}$, this is, for the vector field
$X=\zeta \indel{z}+ \zeta\xi\indel{\zeta}+(\frac{1}{2}\xi^2-\rho\zeta^2)\indel{\xi}$ and~$\widehat{\pi}(z,\zeta,\xi)=z$. Let~$(z(t),\zeta(t),\xi(t))$ be a solution to~$X$. Comparing the projective structures induced by~$z$ and~$t$ in the base, we have
\[ \{t,z(t)\}=-\frac{1}{(z'(t))^2}\{z(t),t\}=\rho(z).\]
It follows from this formula that for the projective structure induced by~$X$, the Christoffel symbol of~$\indel{z}$ is~$\rho$, and coincides, as we sought to establish, with the one induced by the projective connection~$\Xi$.

Notice that the vector field~$X$ is defined on all of~$J^1 T_\mathcal{F}$ and that it is transverse to~$\pi$ over~$M\setminus\mathrm{Sing}(\mathcal{F})$  away from~$\ker(j_0)$.

\begin{rmk} The  conditions~(\ref{sl2}) that the geodesic vector field of a foliated projective structure must satisfy are also sufficient ones. From this, we have an equivalent formulation for our definition:  \emph{a foliated projective structure subordinate to~$\mathcal{F}$ is a vector field~$X$ on~$J^1 T_\mathcal{F}$ that projects onto~$\mathcal{F}$ and that satisfies the relations~(\ref{sl2})} with the vector fields~$H$ and~$Y$.
\end{rmk}

Let~$\pi:\mathbf{P}(J^1T_\mathcal{F})\to M$ be the projectivization of~$J^1 T_\mathcal{F}$.  The foliation by curves induced by~$X$ on~$J^1 T_\mathcal{F}$ is invariant by the flow of~$H$ and, thus, the total space of~$\mathbf{P}(J^1T_\mathcal{F})$  inherits a foliation by curves~$\mathcal{G}$ that projects onto~$\mathcal{F}$. This will be the main object in the proof of Theorem~\ref{ourindthmproj}.

For the expression of~$\mathcal{G}$  in local coordinates,  cover~$U_j\times\mathbf{P}^1$ by charts~$U_j^+=U_j\times\mathbf{C}$ and~$U_j^-=U_j\times\mathbf{C}$, where, in~$U_j^+$ (resp.~$U_j^-$), an affine coordinate~$u_j$ (resp.~$v_j$) for the second factor is given by~$[u_j:1]=[\xi_j:\zeta_j]$ (resp.~$[1:v_j]=[\xi_j:\zeta_j]$). In~$U_j^+$, $\mathcal{G}$ is tangent to the vector field
\begin{equation}\label{locexping}Z_j-\left(\frac{1}{2}u_j^2+\rho_j\right)\del{u_j}\end{equation} 
and, in~$U_j^-$, to~$Z_j+(\frac{1}{2}+\rho_jv_j^2)\indel{v_j}$. These last two glue together into a vector field with isolated singularities on~$U_i^+\cup U_i^{-}$ and, from~(\ref{for:ranktwo}), in~$U_j^+\cap U_i^+$, $g_{ij}u_j=  u_i+Z_j(g_{ij})$.

These calculations show that~$T_\mathcal{G}=\pi^*T_\mathcal{F}$. They also exhibit the fact that, even if~$X$ is not transverse to~$\pi$ along~$\ker(j_0)$, $\mathcal{G}$ is transverse to~$\pi$ above~$M\setminus\mathrm{Sing}(\mathcal{F})$.

For an integral curve~$C$ of~$\mathcal{F}$, the restriction of~$\mathcal{G}$  to~$\pi^{-1}(C)$ is a Riccati foliation with respect to the rational fibration~$\pi|_C$.  There is a section~$\sigma:M\to \mathbf{P}(J^1T_\mathcal{F})$  of~$\pi$ given by  the projectivization of the subbundle~$\ker(j_0)$, which is everywhere transverse to~$\mathcal{G}$ and which inherits, in consequence, a foliated projective structure subordinate to~$\mathcal{F}$. We claim that this projective structure is the one we started with. Let us prove this in dimension one,  in a coordinate where~$Z$ is~$\indel{z}$, where~$\mathcal{G}$ is generated by~$\indel{z}-(\frac{1}{2}\rho(z)+\widetilde{u}^2)\indel{\widetilde{u}}$ for~$\widetilde{u}=\frac{1}{2}u$. The charts of the projective structure induced on the line~$\widetilde{u}=\infty$  by the orbits of this vector field   are the solutions~$h$ of the Schwartzian equation~$\{h,z\}=\rho(z)$~\cite[Prop.~2.1]{loray-marin}. This proves our claim.

\begin{rmk} 
For a projective structure on a curve, the triple~$(\mathbf{P}(J^1T_\mathcal{F}), \mathcal{G}, \sigma)$ gives the  \emph{graph} of the  projective structure (see~\cite[Section~1.5]{loray-marin}). Our construction gives actually more: since~$\mathcal{G}$ comes from the quotient of~$X$, it its naturally endowed with a foliated affine structure, and~$\pi$ identifies the projective classes of the foliated affine structures with the foliated projective structure of~$\mathcal{F}$.\end{rmk}

\subsection{Proof of Theorem~\ref{ourindthmproj}} Let  us  briefly recall Baum and Bott's index theorem in the generality that we will need. Let~$N$ be a compact complex manifold of dimension~$m$, $\mathcal{H}$ a holomorphic foliation by curves  on~$N$ having only finitely many singularities, all of them non-degenerate. We will use the terminology around symmetric polynomials introduced at the beginning of Section~\ref{sec:projindex}. For~$p\in\mathrm{Sing}(\mathcal{H})$, let~$A_p$ be the linear part at~$p$ of a vector field generating~$\mathcal{H}$ in a neighborhood of~$p$. Define~$\sigma_i(A_p)$ by~$\det(\mathbf{I}+tA_p)=\sum_{i=0}^m \sigma_i(A)t^i$ and, for a symmetric polynomial $\varphi(x_1,\ldots,x_m)$,   let~$\varphi(A_p)=\widetilde{\varphi}(\sigma_1(A_p),\ldots,\sigma_m(A_p))$.

Baum and Bott's index theorem~\cite{baum-bott}  affirms that
\begin{equation}\label{eq:gen_bb}
\sum_{p\in\mathrm{Sing}(\mathcal{H})} \frac{\varphi(A_p)}{\det(A_p)}= \varphi( TN- T_\mathcal{H}).
\end{equation} 

Theorem~\ref{ourindthmproj} will follow from applying it to the foliation~$\mathcal{G}$ on~$\mathbf{P}(J^1T_\mathcal{F})$, for the same~$\varphi$ appearing in its statement.

We begin by calculating the left-hand side of~(\ref{eq:gen_bb}) for the foliation~$\mathcal{G}$ on~$\mathbf{P}(J^1T_\mathcal{F})$. Let~$p\in U_j$ be a singular point of~$\mathcal{F}$,  $Z$ a vector field generating~$\mathcal{F}$ in a neighborhood of~$p$, $\rho=\Xi(Z)$ the Christoffel symbol of~$Z$, which, by hypothesis, does not vanish at~$p$. In~$\mathbf{P}(J^1T_\mathcal{F})$, above~$p$, there  are  two singular points of~$\mathcal{G}$. At these, from~(\ref{locexping}), the ratios of the eigenvalues of a vector field tangent to~$\mathcal{G}$ are~$[\lambda_1:\cdots:\lambda_n:\sqrt{-2\rho(p)}]$ and~$[\lambda_1:\cdots:\lambda_n:-\sqrt{-2\rho(p)}]$. The sum of the contributions of these two points to the left-hand side of~(\ref{eq:gen_bb}) is
\begin{multline*}  \frac{\varphi(\lambda_1,\ldots,\lambda_n,\sqrt{-2\rho(p)})}{\lambda_1\cdots\lambda_n\sqrt{-2\rho(p)}}-\frac{\varphi(\lambda_1,\ldots,\lambda_n,-\sqrt{-2\rho(p)})}{\lambda_1\cdots\lambda_n\sqrt{-2\rho(p)}} = \\  =\frac{\varphi(\nu_1,  \ldots,  \nu_n,1)-\varphi(\nu_1,  \ldots,  \nu_n,-1)}{\nu_1  \cdots  \nu_n}=2\frac{\varphi_{\mathrm{odd}}(\nu_1,  \ldots,  \nu_n,1)}{\nu_1  \cdots  \nu_n},
\end{multline*}
where~$\nu_1,\ldots, \nu_n$ are the principal projective ramification indices of the foliated projective structure at~$p$, and where the first equality follows from~(\ref{fromeigentoindex}). This last expression is well-defined (Remark~\ref{rem:even}). Since there are no further singular points of~$\mathcal{G}$, the sum of these   terms over the singular points of~$\mathcal{F}$  gives twice the total sum in the  left-hand side of~(\ref{eq:tracepown}).

Let us now come to   the right-hand side of~(\ref{eq:gen_bb}) for the foliation~$\mathcal{G}$ on~$\mathbf{P}(J^1T_\mathcal{F})$. In order to express this right-hand side in terms of data in~$M$,  we need  a better understanding of the Chern classes of~$\mathbf{P}(J^1T_\mathcal{F})$. Grothendieck's approach~\cite{grothendieck} is particularly well adapted to the study of Chern classes of projective bundles. 

Let~$\widetilde{\pi}:V\to M$ be a vector bundle, $\pi:\mathbf{P}(V)\to M$ the associated projective bundle.  We denote by \(L\rightarrow \mathbf{P}(V) \) the  dual  of the tautological bundle, and by \(\zeta \in H^2 (\mathbf{P}(V), \mathbf{Z}) \) the Chern class of \(L\).   By Grothendieck's definition of Chern classes,  
\begin{equation}\label{eq: Grothendieck} \zeta^2 + \pi^* c_1 (V) \zeta + \pi^* c_2 (V) =0 ,\end{equation}
where \( c_k (V) \in H^{2k} (M,\mathbf{Z})\) is the $k$-th Chern class of \(V\). From the short exact  sequence
\[ 0\to   \ker(D\pi)\longrightarrow T\mathbf{P}(V) \stackrel{D\pi}{\longrightarrow} \pi^*TM\to 0, \]
for the total Chern classes we have
\[c(T\mathbf{P}(V))=c(\ker(D\pi))c(\pi^*TM).\]
\begin{prop}\label{kerdp} We have an isomorphism $\ker (D\pi)    \simeq  L^{\otimes 2} \otimes \det (V)$. In particular, $c(\ker{D\pi})=1+2\zeta+c_1(\det(V))$.
\end{prop}
\begin{proof} We have the following canonical isomorphism:  given \(x\in M\) and \(l \in \mathbf P (V_x)\) (a line in \(V_x\)) we have 	
\begin{equation} \label{eq: tangent bundle P1} T_l \mathbf P (V_x) \simeq \mathrm{Hom} (l , V_x / l ).
\end{equation}
Indeed, the derivative of the projectivization \(p : V_x \setminus \{0\}  \rightarrow \mathbf P (V_x) \) induces for each \(u\in l \) an isomorphism \( Dp _u : V_x / l \rightarrow T_x \mathbf P (V_x) \)  that satisfies 
\begin{equation} \label{eq: homogeneity} \lambda  Dp_{\lambda u } = Dp_u \text{ for any } u\in l \text{ and } \lambda \in \mathbf{C}, \end{equation}
since \(p\) is invariant by multiplication by \(\lambda\in \mathbf{C}\). The isomorphism \eqref{eq: tangent bundle P1} is then defined by the formula 
\[ v \in T_l \mathbf P(V_x) \text{ corresponds to } u \in l \stackrel {\varphi_v}{\longrightarrow} (Dp_u) ^{-1} (v) \in V_x/ l, \]
and equation \eqref{eq: homogeneity} shows that \(v\mapsto \varphi_v \) is linear.
	
Given \( \omega\in \wedge ^2 V_x ^* \) (a dual of the determinant bundle) and \( \psi  \in \mathrm{Hom} (l , V_x / l )\) we can form the quadratic polynomial \(\varphi \) on \(l\) (an element of \( L^{\otimes 2} \)) by the formula~$\varphi ( u ) = \omega( u , \psi(u) )$. This operation produces the desired isomorphism of line bundles over~\(\mathbf P (V)\).\end{proof}

Associated to~$\pi:\mathbf{P}(V)\to M$	we have the \emph{integration along the fibers} (or \emph{transfer})  map $\pi^!:H^l(\mathbf{P}(V),\mathbf{Z})\to H^{l-2}(M,\mathbf{Z})$~\cite[Ch.~VIII]{dold}. It satisfies the  product formula $\pi^!(\alpha\cdot \pi^*\beta)=\pi^!\alpha\cdot\beta$  as well as the Fubini  relation $\langle\alpha,[\mathbf{P}(V)]\rangle=\langle\pi^!\alpha,[M]\rangle$ (here, $\langle,\rangle$ denotes the cohomology-homology pairing and~$[\cdot]$ the fundamental class). We have~$\pi^!\zeta=1$. In particular,
\begin{equation}\label{umkehr}\langle \zeta\cdot \pi^*\beta,[\mathbf{P}(V)]\rangle= \langle \pi^!(\zeta\cdot \pi^*\beta),[M]\rangle= \langle  \pi^!\zeta\cdot\beta,[M]\rangle= \langle  \beta,[M]\rangle. \end{equation}

In our setting, $J^{1}T_\mathcal{F}$  is an extension of~$T_\mathcal{F}$ by the trivial bundle, so~$\det(J^{1}T_\mathcal{F})\simeq T_\mathcal{F}$ and,  for the total Chern class, $c(J^{1}T_\mathcal{F})=c(T_\mathcal{F})$, this is,
$c_1(J^{1}T_\mathcal{F})=c_1(T_\mathcal{F})$  and  $c_2(J^{1}T_\mathcal{F})=0$. From~(\ref{eq: Grothendieck}), for~$V=J^{1}T_\mathcal{F}$, $\zeta^2 = - \pi^* c_1 (T_\mathcal{F})  \zeta$, and, for all~$k\geq 1$,
\begin{equation} \label{eq: iteration} \zeta^{k} = (- \pi^* c_1 (T_\mathcal{F}) ) ^{k-1} \zeta.  \end{equation}

From the previously established identification~$T_\mathcal{G}=\pi^{*}T_\mathcal{F}$, 
\begin{multline}\label{cherndecomp} c(T\mathbf{P}(J^1T_\mathcal{F})-T_\mathcal{G})=\frac{c(T\mathbf{P}(J^1T_\mathcal{F}))}{c(T_\mathcal{G})}=\frac{c(\ker (D\pi))c(\pi^*TM)}{c(\pi^*T_\mathcal{F})}  \\ =(1+c_1(\ker(D\pi)))\pi^*c(TM-T_\mathcal{F}).
\end{multline}
If, as in~(\ref{symdecomp}), $\varphi(x_1,\ldots,x_{n+1})=\sum_{i=0}^{n+1} x_{n+1}^i\widehat\varphi_{n+1-i}(x_1,\ldots,x_{n})$, from~(\ref{cherndecomp}) we have, setting~$\kappa=c_1(\ker(D\pi))$, 
\begin{equation}\label{develop}\varphi(c(T\mathbf{P}(J^1T_\mathcal{F})-T_\mathcal{G}))=\sum_{i=0}^{n+1} \kappa^i\cdot \widehat\varphi_{n+1-i}(\pi^*c(TM-T_\mathcal{F})).\end{equation}
From the equality~$\kappa=2\zeta+\pi^*c_1(T_\mathcal{F})$ established in Proposition~\ref{kerdp} and from~(\ref{eq: iteration}), $\kappa^2=\pi^*c_1^2(T_\mathcal{F})$, and thus
\[ \kappa^i=\begin{cases}  \pi^*c_1^i(T_\mathcal{F}) +2\zeta\cdot \pi^*c_1^{i-1}(T_\mathcal{F})  & \text{if }  i \text{ is odd,} \\  \pi^*c_1^i(T_\mathcal{F})   & \text{if }   i \text{ is even.}  \end{cases}\]
Hence, (\ref{develop}) equals 
\[\pi^*\sum_{i=0}^{n+1} c_1^{i}(T_\mathcal{F})   \widehat\varphi_{n+1-i}( c(TM-T_\mathcal{F})) +2\zeta\cdot \pi^* \sum_{j=0}^{\lfloor n/2\rfloor }c_1^{2j}(T_\mathcal{F})   \widehat\varphi_{n-2j} (c(TM-T_\mathcal{F})),\]
but the first summand is trivial, since it is the pull-back of classes in~$M$ whose degree   exceeds the dimension of~$M$. We conclude that, on~$\mathbf{P}(J^1T_\mathcal{F})$,
\[   \varphi(c(T\mathbf{P}(J^1T_\mathcal{F})-T_\mathcal{G}))   = 2  \zeta \cdot\pi^*\sum_{j=0}^{\lfloor n/2\rfloor }c_1^{2j}(T_\mathcal{F})   \widehat\varphi_{n-2j} (c(TM-T_\mathcal{F}))  .\]
By~(\ref{umkehr}), this expression equals twice the right-hand side of~(\ref{eq:tracepown}).  This finishes the proof of Theorem~\ref{ourindthmproj}.

\section{Regular foliations}

The index theorems of the previous sections   impose severe restrictions on foliated affine and projective structures along regular foliations. On surfaces, they will allow for a full classification of these structures.

\subsection{Some consequences of the index theorems} For a compact surface, the existence of a regular foliation   supporting a foliated projective structure greatly limits its topology:

\begin{corollary}\label{zerosign} A compact complex surface admitting a regular foliation that supports a foliated projective structure has vanishing signature. \end{corollary}
\begin{proof}If~$M$ is a compact complex surface that  admits a regular foliation~$\mathcal{F}$, a consequence  of the Baum-Bott index theorem is that  $c_1^2(T_\mathcal{F})=c_1^2(M)-2c_2(M)$~\cite[Section~2]{brunella-nonsing}. (From Hirzebruch's formula, the signature~$\tau(M)$ equals~$\frac{1}{3}(c_1^2(M)-2c_2(M))$~\cite[Ch.~I, \S 3]{bphv}.) If such an~$\mathcal{F}$ supports a foliated projective structure,  from the instance of Theorem~\ref{ourindthmproj} in Example~\ref{thmexasimplest}, $c_1^2(T_\mathcal{F})=0$, and the signature of~$M$ vanishes.
\end{proof}

Since Kodaira fibrations  have non-vanishing signature~\cite[Ch.~V, \S 14]{bphv}, this gives another proof of the fact that they do not support foliated projective structures.

This obstruction, together with the classification of regular foliations on surfaces provided by Brunella~\cite{brunella-nonsing}, permits to list all the regular foliations on complex surfaces that admit foliated projective structures. (In section~\ref{sec:class2dreg} we will give a classification of the affine and projective structures on these foliations.)  
\begin{corollary}\label{c: classification} The regular foliations on  compact complex surfaces that admit   foliated projective structures are: isotrivial   fibrations, suspensions, linear foliations on tori, turbulent foliations, evident foliations on Hopf or Inoue surfaces,  and evident foliations in quotients of the bidisk. \end{corollary}

\begin{proof} Regular foliations on surfaces were classified by Brunella~\cite[Thm.~2]{brunella-nonsing}. Other than the ones in the previous list,  there are  non-isotrivial fibrations and  some transversely hyperbolic foliations with dense leaves. But non-isotrivial fibrations are necessarily Kodaira ones, since regular elliptic fibrations are necessarily isotrivial, and among the transversely hyperbolic foliations with dense leaves, those supported on surfaces of vanishing signature are quotients of the bidisk (see the closing remarks in~\cite{brunella-nonsing}). 
	
There remains to exhibit foliated projective structures for all the foliations in the above list. Linear foliations on tori are tangent to holomorphic vector fields, and have a foliated translation structure.  Rational fibrations have foliated projective structures (say, by Savel'ev's theorem~\cite{savelev}), elliptic ones carry foliated affine structures by the results in Example~\ref{ex: elliptic fibrations}, and fibrations of higher genus supporting foliated projective structures are   isotrivial, and hence have foliated projective structures (e.g. the complete hyperbolic ones along the fibers).  The existence of foliated affine or  projective structures for suspensions, elliptic fibrations, turbulent foliations, and Hopf and Inoue surfaces has already been addressed in Examples~\ref{ex:susp}, \ref{ex:turb}, \ref{exam-inoue} and~\ref{exam-hopf}.    
\end{proof}

Corollary \ref{c: classification} shows that a regular foliation on an algebraic  compact complex surface carries a foliated projective structure if and only if it is not a foliation of general type.  One direction can be directly proved, more generally,  for all manifolds of even dimension:

\begin{prop}\label{typgen-noproj} On a compact algebraic manifold of even dimension, a regular foliation of general type cannot support a foliated projective structure.  
\end{prop}
\begin{proof} Let~$M$ be the manifold, $n$ its dimension, $\mathcal{F}$ the foliation. The general type assumption on \(\mathcal F\) says that \(K_{\mathcal F}\) is big, namely 
\[\frac{\log h^0(K_{\mathcal F}^{\otimes m}) }{\log m} \rightarrow n.\]
By \cite[Theorem 2, p.~51]{McQuillan}, \(K_\mathcal F\)  is also nef.  A nef and big line bundle \(L\) on an algebraic variety of dimension \(n\) satisfies \(c_1^n(L)  >0\); this can be deduced from the asymptotic Riemann-Roch formula, stating that for a nef line bundle~\(L\),
\[h^0 (L^m) = \frac{1}{n !}c_1^n(L)  m^n +O(m^{n-1}),\]
see e.g. \cite[Corollary~1.4.41]{Lazarsfeld}.  Applying this to \(L=K_{\mathcal F}\) gives \( (-1)^n c_1^n(T_{\mathcal F})  = c_1^n(K_{\mathcal F}) >0\).  On the other hand, if the manifold has even dimension and admits a foliated projective structure, by the particular case of Theorem~\ref{ourindthmproj} described in Example~\ref{thmexasimplest}, $c_1^n(T_\mathcal{F})=0$. 
\end{proof}

In higher dimensions, there are regular foliations which are not of general type, but which do not support foliated projective structures. Take for instance the product of a Kodaira fibration on a surface by a curve, producing a fibration which is not of general type (as a foliation), but which does not have any foliated projective structure. 	

We do not know if Proposition~\ref{typgen-noproj} holds true in odd dimensions. We have nevertheless the following weak version of it.

\begin{prop}\label{typgen-noaff} On a compact algebraic manifold, a regular foliation of general type cannot support a foliated affine structure.  
\end{prop} 

\begin{proof}	
	In the presence of a foliated affine structure, it follows from Theorem~\ref{ourdescartes} that~$c_1^n(T_\mathcal{F})=0$, the left hand side vanishing by the absence of singular points. The arguments in the proof of Proposition~\ref{typgen-noproj} allow to conclude.
\end{proof}

Many families of  regular foliations are given by  \emph{characteristic foliations} on hypersurfaces of general type in compact symplectic manifolds. They are those generated by the distribution given by   the kernel of the restriction of the symplectic form to the hypersurface, see e.g. \cite{HwangViehweg}.   By the adjunction formula  and the  fact that the top wedge power of the normal bundle of the foliation is trivial, the canonical bundle of such a foliation is isomorphic to the canonical bundle of the hypersurface, so the characteristic foliation is of general type, and does not carry a foliated affine structure by Proposition \ref{typgen-noaff}. We do not know if these foliations admit  foliated projective structures;  our index formulae do not give any obstructions whatsoever in this case (we leave to the reader to check that,  for all these foliations and in all instances of Theorem~\ref{ourindthmproj},  the right-hand side gives always zero).

These foliations occur in odd dimensions. In even ones, beyond the case of surfaces, we do not know the extent to which our index formulae give relevant obstructions for the existence of foliated projective structures. We do not seem to have enough examples of regular foliations on   manifolds of even dimension.  

\subsection{A classification of  foliated affine and projective structures for regular foliations on surfaces} \label{sec:class2dreg}

We will now classify the foliated affine and projective structures  for the foliations appearing in Corollary~\ref{c: classification}. We begin with  the following Lemma.
\begin{lemma}\label{prop:def-vf} Let $M$ be a compact manifold, $X$ a nowhere vanishing vector field on~$M$, $\mathcal{F}$ the regular foliation induced by~$X$. 
	\begin{itemize}
		\item The spaces of foliated affine and projective structures of~$\mathcal{F}$ are both one-dimensional. The natural map from affine to projective structures is a double cover ramified at the affine structure induced by~$X$.
		\item  Assume~$\sigma:M\to M$ is a fixed-point-free involution such that~$D\sigma(X)=-X$, and let~$N=M/\sigma$ and~$\mathcal{G}$  be  the foliation on~$N$ induced by~$X$.  Then  the space of foliated projective structures on~$\mathcal{G}$ is one-dimensional, and the only foliated affine structure on~$\mathcal{G}$ is the one induced by~$X$. 	
	\end{itemize}
\end{lemma}	
\begin{proof}  A foliated translation structure along~$\mathcal{F}$ is induced by~$X$, and both the space  of foliated affine structures and that of foliated projective ones are non-empty. Let~$\nabla$ be the connection of a foliated affine structure on~$\mathcal{F}$. The holomorphic function~$\nabla(X)$ is a constant from which~$\nabla$ can be recovered. In the same way, foliated projective connections can be recovered from the corresponding Christoffel symbols of~$X$. The last assertion of the first part follows from formula~(\ref{for:affaproj}).
	
For the second part, notice that, from formula~(\ref{conn-fun}), $\nabla(\sigma_*X)=\nabla(-X)=-\nabla(X)$, that the affine structure associated to~$\nabla$ will be preserved by~$\sigma$ if and only if~$\nabla(X)\equiv 0$, this is, if it is the one induced by~$X$. If~$\Xi$ is a foliated projective connection, from formula~(\ref{christproj}), $\Xi(\sigma_*X)=\Xi(-X)=\Xi(X)$, and every foliated projective structure along~$\mathcal{F}$ is preserved by the involution, inducing one on~$\mathcal{G}$. Reciprocally, every foliated projective structure on~$\mathcal{G}$ induces one on~$\mathcal{F}$.
\end{proof}

Let us now come to the classification.

\subsubsection{Linear foliations on tori}  Foliated affine and projective structures  are induced by nowhere vanishing vector fields and are covered by the first part of Lemma~\ref{prop:def-vf}.

\subsubsection{Inoue surfaces} The existence of foliated affine structures on these surfaces was  discussed in Example~\ref{exam-inoue}.

Consider an Inoue surface~$S_M$~\cite[\S 2]{inoue}. It is the quotient of~$\mathbf{H}\times \mathbf{C}$ under the action of a semidirect product~$\mathbf{Z}\ltimes \Gamma$; let~$\pi:\mathbf{H}\times \mathbf{C}\to S_M$ denote the quotient map. Let~$\mathcal{F}$ be either the vertical or the horizontal foliation on~$S_M$, and let~$\omega$  be a section of~$K_\mathcal{F}^2$. The preimage of~$\mathcal{F}$ in~$\mathbf{H}\times \mathbf{C}$ is generated by a coordinate vector field~$X$. The contraction of~$X^{\otimes 2}$ with~$\pi^*\omega$ gives a holomorphic function on~$\mathbf{H}\times \mathbf{C}$ which is invariant under the action of~$\Gamma$ and which,  by~\cite[Lemma~3]{inoue}, is constant. In particular, the function~$\pi^*\omega(X^{\otimes 2})$ descends to~$S_M$. If it were not the zero constant, $X^{\otimes 2}$ would descend to~$S_M$ as well, but $S_M$ has no holomorphic vector fields~\cite[Prop.~2]{inoue} and neither do its double covers, which are Inoue surfaces of the same kind.   We conclude that~$\omega$ vanishes identically, that~$K_\mathcal{F}^2$ has no nonzero sections, and that the natural foliated affine structures  are rigid as projective ones (a variation of these arguments shows that~$h^0(K_\mathcal{F})=0$, that they are also rigid as affine ones).

Consider now an Inoue surface~$S^{(+)}$~\cite[\S 3]{inoue}. It is the quotient of~$\mathbf{H}\times \mathbf{C}$ under the action of a group that preserves the coordinate vector field on the second factor, and which induces a nowhere-vanishing vector field~$X$ on~$S^{(+)}$. The foliated affine and projective structures on the induced foliation are described by the first part of Lemma~\ref{prop:def-vf}.

Lastly, consider an Inoue surface~$S^{(-)}$~\cite[\S 4]{inoue}. It has an Inoue surface of type~$S^{(+)}$ as an unramified double cover~$\rho:S^{(+)}\to S^{(-)}$ induced by a fixed-point free involution of~$S^{(+)}$ acting upon~$X$ by changing its sign. The second part of Lemma~\ref{prop:def-vf} classifies foliated structures on the associated foliation.

By the results in~\cite{brunella-nonsing}, these foliations are the only ones on Inoue surfaces, and the above arguments give a complete classification of foliated affine and projective structures on them.

\subsubsection{Hopf surfaces}  Let~$S$ be a Hopf surface, quotient of~$\mathbf{C}^2\setminus\{0\}$ under the action of a group~$G$ containing a contraction~$g$, $\mathcal{F}$ a foliation on~$S$, as described in Example~\ref{exam-hopf}. (By the results in~\cite{brunella-nonsing}, these are all the foliations on Hopf surfaces.) Let us follow the notations of Example~\ref{exam-hopf}.

If~$S$ is a primary Hopf surface (if~$G$ is generated by~$g$) and~$\mathcal{F}$ is tangent to a nowhere-vanishing vector field~$X$ (like those induced by linear and ``Poincar\'e-Dulac'' ones), the first part of Lemma~\ref{prop:def-vf} gives a complete description.  For these foliations, for the associated secondary Hopf surfaces, following Kato's classification~\cite{kato-hopf, kato-hopf-err}, the action of~$G$ will preserve the vector field~$X$ (and the situation is covered by the first part of Lemma~\ref{prop:def-vf}) except in the  situation which we now describe. 

Let~$S$ be the primary elliptic Hopf surface with~$g$ of the form~$(x,y)\mapsto(\alpha x,\alpha y)$.  Let~$X$ be the vector field on~$S$ induced by~$Ay\indel{x}+Bx\indel{y}$ ($A,B\neq 0$); it generates a turbulent foliation~$\mathcal{F}$ with respect to the elliptic fibration induced by~$(x,y)\mapsto [x:y]$. Consider the group~$G$ of transformations of~$\mathbf{C}^2\setminus\{0\}$ that contains~$g$, a (possibly trivial) subgroup~$H$ of linear diagonal elements of finite order plus a nontrivial  element~$\rho$ of finite order of the form~$(x,y)\mapsto (\mu x,-\mu y)$ (notice that~$\rho^2\in H$). Let~$S'$ be the quotient of~$\mathbf{C}^2\setminus\{0\}$  under the action of the subgroup of~$G$ generated by~$g$ and~$H$;  it has a vector field~$X'$ induced by~$X$.  Let~$S''$ be the quotient of~$\mathbf{C}^2\setminus\{0\}$ under the full action of~$G$. The  surface~$S''$  is the quotient of~$S'$ under an involution induced by~$\rho$, which acts upon~$X'$ by changing its sign,  inducing the foliation~$\mathcal{F}''$ on~$S''$.    The second part of Lemma~\ref{prop:def-vf} classifies the foliated structures on~$\mathcal{F}''$.

Consider now a foliation~$\mathcal{F}$ on the primary Hopf surface~$S$ induced by~$\indel{x}$, let~$\mathcal{H}$ be the associated foliation on~$\mathbf{C}^2\setminus\{0\}$ and~$\rho$ be the section of~$K_\mathcal{H}^2$ whose contraction with~$(\indel{x})^{\otimes 2}$ is the constant~$1$. Let~$\omega$ be a section of~$K_\mathcal{F}^2$. Let~$F$ be the holomorphic function on~$\mathbf{C}^2\setminus\{0\}$ such that~$\pi^{*}\omega=F\rho$. It satisfies the relation~$\alpha^2 F(\alpha x+\lambda y^n,\beta y)=F(x,y)$, and, reciprocally, every~$F$ satisfying this relation gives  a section of~$K_\mathcal{F}^2$. This implies,  if~$\lambda=0$, that~$x^2F(x,y)$ is invariant,  or, if~$\lambda\neq 0$, that~$y^{2n}F(x,y)$ is.  This holomorphic function descends to~$S$ and is thus constant, and this implies that the holomorphic function~$F$ has poles  unless it vanishes identically. This proves that~$h^0(K^2_\mathcal{F})=0$. In essentially the same way, $h^0(K_\mathcal{F})=0$:  the foliated affine structures are rigid both as affine and as projective ones. This rigidity is inherited by the foliated structures induced on the associated secondary Hopf surfaces.

This gives a complete classification of foliated affine and projective structures on Hopf surfaces. 

\subsubsection{Quotients of the bidisk} In the quotient of~$\mathbf{D}\times\mathbf{D}$ under the action of a lattice  in~$\mathrm{Aut}(\mathbf{D} \times \mathbf{D})$ which is not virtually a product, the vertical (or horizontal) foliation~$\mathcal{F}$ carries, by construction, a foliated projective structure  which is not affine. For this foliation, $\mathrm{kod}(\mathcal{F})=-\infty$~\cite[Ch.~9, Section~5]{brunella}, and, in particular, $h^0(K^2_\mathcal{F})=0$: the foliated projective structure is a rigid one.

\subsubsection{Turbulent foliations} In this case, the results of Example~\ref{ex:turb} already give a classification of foliated affine and projective structures on them. For instance, in a turbulent foliation where in the  tangency divisor of the foliation and the fibration the non-transverse fibers appear simply, the foliated affine (resp. projective) foliated   structures are in correspondence with the   affine (resp. projective) structures on the base which have Fuchsian singularities at the points corresponding to the non-transverse fibers.

\subsubsection{Regular elliptic fibrations}Let~$\pi:S\to C$ be an elliptic fibration without singular fibers. Let~$Z$ be a vector field tangent to~$\pi$ like those constructed in Example~\ref{ex: elliptic fibrations}, $D_Z$ its associated divisor. By construction, $-D_Z$ is the divisor of the \emph{relative canonical bundle} $K_{S/C}=K_S\otimes \pi^*K_C^*$ (see~\cite[Ch.~2, Section~3]{brunella}); it can be very explicitly described~\cite[Ch.~V, Section~12]{bphv}.

By formula~(\ref{conn-fun}), through their evaluation on~$Z$,  the foliated connections on~$T_\mathcal{F}$ are in correspondence with the meromorphic functions~$g$ on~$C$ which are either identically zero or which, denoting by~$(\cdot)$ the divisor of a meromorphic function, are such that~$(\pi^*g)\geq D_Z$. Similarly, by~(\ref{christproj}), the projective connections on~$\mathcal{F}$ are in correspondence with the meromorphic functions~$g$ on~$C$ for which~$(\pi^*g)\geq 2D_Z$.

\subsubsection{Isotrivial fibrations of higher genus} Let~$\pi:S\to C$ be a regular isotrivial fibration of higher genus with typical fiber~$F$. Since neighboring fibers are canonically identified,  there is a globally defined foliation  transverse to~$\pi$ realizing this identification. Its global holonomy representation gives a group~$\Gamma\subset \mathrm{Aut}(F)$. When pulled-back to the finite covering of \(C\) defined by the kernel of the holonomy representation, the fibration becomes a product. Since, on a product, foliated projective structures are constant (such a structure is given by a map from the base of the fibration to the moduli of projective structures on a fixed curve, which is affine), the foliated  projective structures on the fibration  are in correspondence with the projective structures on~$F$ invariant by~$\Gamma$.


\providecommand{\bysame}{\leavevmode\hbox to3em{\hrulefill}\thinspace}
\providecommand{\href}[2]{#2}

\end{document}